\documentclass[english]{amsart}

\usepackage{amsmath} 
\usepackage{pdfsync}
\usepackage{amssymb} 

\usepackage{amscd} 
\usepackage{tabularx} 

\usepackage[all,cmtip,line]{xy}

\newcommand{\F}{\mathbb{F}} 
\newcommand{\lambdabar}{\overline{\lambda}}

\DeclareMathOperator{\Frob}{Frob} 

\newcommand{\mcal}{\mathcal}
\newcommand{\To}{\longrightarrow} 
\newcommand{\into}{\hookrightarrow} 
\newcommand{\onto}{\twoheadrightarrow} 
\newcommand{\isoto}{\stackrel{\sim}{\To}} 
\newcommand{\rec}{\operatorname{rec}}
\newcommand{\Art}{\operatorname{Art}}
\newcommand{\nind}{\operatorname{n-Ind}}

\newcommand{\bigO}{\mathcal{O}}

\newcommand{\Z}{\mathbb{Z}} 
\newcommand{\A}{\mathbb{A}} 
\newcommand{\Q}{\mathbb{Q}} 
 
\newcommand{\T}{\mathbb{T}} 
\newcommand{\gn}{\mathcal{G}_n}

\newcommand{\C}{\mathbb{C}}

\newcommand{\rhobar}{\overline{\rho}} 
\newcommand{\taubar}{\overline{\tau}}

\newcommand{\Gal}{\operatorname{Gal}} 
\newcommand{\GL}{\operatorname{GL}} 
\newcommand{\PGL}{\operatorname{PGL}}

\newcommand{\Qbar}{\overline{\Q}} 
\newcommand{\Qp}{\Q_p} 
\newcommand{\mf}{\mathfrak}
\newcommand{\Qpbar}{\overline{\Q}_p} 
\newcommand{\Fpbar}{\overline{\F}_p} 
\newcommand{\Fbar}{\overline{\F}} 
\newcommand{\rbar}{\overline{r}} 
\newcommand{\tv}{{\tilde{v}}} 
\newcommand{\gfv}{{G_{F_{\tilde{v}}}}} 
\newcommand{\ind}{\operatorname{Ind}_{KZ}^G} 
\newcommand{\cind}{\operatorname{c-Ind}_{KZ}^G}

\newcommand{\Res}{\operatorname{Res}}

\newcommand{\Spec}{\operatorname{Spec}} 
\newcommand{\Ind}{\operatorname{Ind}} 
\newcommand{\SL}{\operatorname{SL}} 
 
\newcommand{\ad}{\operatorname{ad}}

\newcommand{\End}{\operatorname{End}} 
\newcommand{\Hom}{\operatorname{Hom}} 
\newcommand{\Map}{\operatorname{Map}} 
\newcommand{\cS}{\mathcal{S}}

\newcommand{\Ext}{\operatorname{Ext}} 
\newcommand{\Aut}{\operatorname{Aut}} 
\newcommand{\Lie}{\operatorname{Lie}}

\newcommand{\Sym}{\operatorname{Sym}}

\usepackage{amsthm}

\newtheorem{thm}{Theorem}[subsection] 
\newtheorem{corollary}[thm]{Corollary} 
\newtheorem{lemma}[thm]{Lemma} 
\newtheorem{prop}[thm]{Proposition} 
\newtheorem{conj}[thm]{Conjecture} \theoremstyle{definition} 
\newtheorem{defn}[thm]{Definition} \theoremstyle{remark} 
\newtheorem{rem}[thm]{Remark} \numberwithin{equation}{subsection}

\begin{document} 
\title{Automorphic lifts of prescribed types} 
\author{Toby Gee} \email{toby.gee@ic.ac.uk} \address{Department of Mathematics, Imperial College London}

\begin{abstract}
	We prove a variety of results on the existence of automorphic Galois representations lifting a residual automorphic Galois representation. We prove a result on the structure of deformation rings of local Galois representations, and deduce from this and the method of Khare and Wintenberger a result on the existence of modular lifts of specified type for Galois representations corresponding to Hilbert modular forms of parallel weight 2. We discuss some conjectures on the weights of $n$-dimensional mod $p$ Galois representations. Finally, we use recent work of Taylor to prove level raising and lowering results for $n$-dimensional automorphic Galois representations.
\end{abstract}
\maketitle 
\tableofcontents 
\section{Introduction}
\subsection{}If $f$ is a cuspidal eigenform, there is a residual representation $$\rhobar_f:G_\Q\to\GL_2(\Fpbar)$$ attached to $f$, and one can ask which other cuspidal eigenforms $g$ give rise to the same representation; if one believes the Fontaine-Mazur conjecture, this is equivalent to asking which geometric representations lift $\rhobar_f$ (here ``geometric'' means unramified outside of finitely many primes and potentially semi-stable at $p$). These questions amount to issues of level-lowering and level-raising (at places other than $p$), and to determining the possible Serre weights of $\rhobar_f$ (at $p$). 

In recent years there has been a new approach to these questions, via the use of lifting theorems due to Ramakrishna and Khare-Wintenberger, together with modularity lifting theorems. In this paper, we present several new applications of this method.

In section \ref{local}, we prove a result about the structure of the local deformation rings corresponding to a mod $p$ representation of the absolute Galois group of a finite extension of $\Q_l$, where $l\neq p$. The method of proof is very close to that of \cite{kis06}, where the corresponding (harder) results are proved for the case $l=p$. This result is precisely the input needed to allow one to prove the existence of modular lifts of specified Galois type at places other than $p$, and we apply it repeatedly in the following sections.

In section \ref{1} we consider the case of Hilbert modular forms. We are able to prove a strong result on the possible lifts of parallel weight $2$, under the usual hypotheses for the application of the Taylor-Wiles method, and a hypothesis on the existence of ordinary lifts; in particular, one can deduce level-lowering results in full generality at places not dividing $p$ (assumed odd). We then apply this result in section \ref{serre} to improve on some of our results from \cite{gee053}  (where we use results from this paper to prove many cases of a conjecture of Buzzard, Diamond and Jarvis on the possible Serre weights of mod $p$ Hilbert modular forms). While we are not able to prove the conjectures of Buzzard, Diamond and Jarvis in full generality, we are able to in the case where $p$ splits completely, and the usual Taylor-Wiles hypothesis holds; this is needed in \cite{kis062}. We also take the opportunity to discuss some generalisations of existing conjectures on Serre weights, which are suggested to us by our work on Hilbert modular forms. For example, we state a general conjecture about the possible weights of an $n$-dimensional representation of $G_\Q$. We hope to discuss these conjectures more fully in future work.

Finally, in section \ref{unitary} we apply these methods to $n$-dimensional Galois representations. Here we make use of the recent work of Taylor (see \cite{tay06}). We are able to prove level-raising and level-lowering results at places away from $p$. Note that such results were originally thought to be needed to prove $R=T$ theorems in this context, but were circumvented in \cite{tay06}. Our proof relies crucially on this work, so does not give a new proof of these $R=T$ theorems. Our results in this final section are more limited than in the $2$-dimensional case, because we do not know any $R=T$ theorems over number fields ramified at $p$. In particular, we cannot at present generalise our approach from \cite{gee053} to this setting. However, the framework established here would suffice to prove such results if such $R=T$ theorems were proved.

We would like to thank Florian Herzig and Mark Kisin for helpful
conversations, and the mathematics department of Harvard University
for its hospitality while much of this paper was written. We would
like to thank Kevin Buzzard, Wansu Kim and James Newton for their
comments on an earlier draft. We would also like to thank the
anonymous referee for a close reading, and many helpful comments and
corrections.
\subsection{Notation}If $K$ is a field, we let $G_K$ denote its
absolute Galois group. We fix an algebraic closure $\Qbar$ of $\Q$,
and regard all algebraic extensions of $\Q$ as subfields of
$\Qbar$. For each prime $p$ we fix an algebraic closure $\Qpbar$ of
$\Qbar$ and an embedding $\Qbar\into\Qpbar$, so that if $F$ is a
number field and $v$ is a finite place of $F$ we have a homomorphism
$G_{F_v}\into G_F$. If $K$ is a finite extension of $\Qp$, we let
$I_K$ denote the inertia subgroup of $G_K$, and if $L$ is a finite Galois
extension of $K$ we write $I_{L/K}$ for $I_K/I_L$.

\section{Local structure of deformation
  rings}\label{local}\subsection{}In this section we prove some
results on the local structure of deformation rings corresponding to a
mod $p$ representation of the absolute Galois group of a finite
extension of $\Q_l$, where $l\neq p$. These results are the analogue
of those proved in section 3 of \cite{kis06} in the case $l=p$, and
the proofs are almost identical (but simpler in our case). In
\cite{kis06} one considers deformations of weakly admissible modules,
whereas in our case we consider deformations of Weil-Deligne
representations; the main difference is that in our case one has a
linear (rather than a semi-linear) Frobenius, and we have no analogue
of the Hodge filtration. The absence of such a filtration simplifies
the computation of the dimensions of our deformation rings. We follow
section 3 of \cite{kis06} extremely closely, only noting the changes
to the arguments of \emph{loc. cit.} that are necessary in our
setting.

Let $K/\Q_l$ ($l\neq p$) be a finite extension. Suppose that the residue field of $K$ has cardinality $l^f$. Fix $d$ a positive integer. As in \cite{kis06}, we use the language of groupoids, which is explained in the appendix to \cite{kis04}.

We firstly recall some results from \cite{fonl}. By a theorem of
Grothendieck, a continuous representation $\rho:G_K\to\GL_d(E)$, $E$ a
finite extension of $\Q_p$, is automatically potentially semi-stable,
in the sense that there is a finite extension $L/K$ such that
$\rho|_{I_L}$ is unipotent. Let $W_K$ denote the Weil group of $K$,
and $WD_K$ the Weil-Deligne group. Denote by $\underline{Rep}_E(WD_K)$
the category of finite dimensional $E$-linear representations of
$WD_K$. One can view an object of this category as a triple
$(\Delta,\rho_0,N)$, where $\Delta$ is a finite dimensional $E$-vector
space, $\rho_0:W_K\to\Aut_E(\Delta)$ is a homomorphism whose kernel
contains an open subgroup of $I_K$, and $N:\Delta\to\Delta$ is an
$E$-linear map satisfying $$\rho_0(w)N=l^{\alpha(w)}N\rho_0(w)\text{
  for all }w\in W_K.$$Here $\alpha:W_K\to\Z$ is the map sending $w\in
W_K$ to $\alpha(w)$ such that $w$ acts on the residue field
$\overline{k}$ of $\overline{K}$ as $\sigma^{\alpha(w)}$, where
$\sigma$ is the absolute Frobenius.

Now, fix $\Phi\in W_K$ with $\alpha(\Phi)=-f$. Fix a finite Galois
extension $L/K$. Then it is easy to see that the full subcategory of
$\underline{Rep}_E(WD_K)$ whose objects are triples
$(\Delta,\rho_0,N)$ as above with $\rho_0|_{I_L}$ trivial is
equivalent to the category $\underline{Rep}_{E,L}(WD_K)$ of triples
$(\Delta,\phi,N)$, where $\Delta$ is a finite dimensional $E$-vector
space with an action of $I_{L/K}$, $\phi\in\Aut_E(\Delta)$, and
$N\in\End_E(\Delta)$, such that $N$ commutes with the action of
$I_{L/K}$, $N\phi=l^f\phi N$, and for all $\gamma\in I_{L/K}$ we have
$\Phi\gamma\Phi^{-1}=\phi\gamma\phi^{-1}$ in $\Aut_E(\Delta)$ (this equivalence follows
from the fact that if $\psi\in W_K$ then $\psi\Phi^{\alpha(\psi)/f}\in
I_K$).

After making a choice of a compatible system of $p$-power roots of unity in $\overline{K}$, we see from Propositions 2.3.4 and 1.3.3 of \cite{fonl} that there is an equivalence of categories between the category of $E$-linear representations of $G_K$ which become semi-stable over $L$, and the full subcategory of $\underline{Rep}_{E,L}(WD_K)$ whose objects are the triples $(\Delta,\phi,N)$ such that the roots of the characteristic polynomial of $\phi$ are $p$-adic units (such an equivalence is given by the functor $\widehat{\underline{WD}}_{pst}$ of section 2.3.7 of \cite{fonl}). We will refer to such a triple as an \emph{admissible} triple. This is not standard terminology. Note for future reference that an extension of admissible objects is again admissible. For the rest of this section we will freely identify Galois representations with their corresponding Weil-Deligne representations.

Firstly we define two groupoids on the category of $\Q_p$-algebras $A$. Let $\mathfrak{Mod}_N$ be the groupoid whose fibre over a $\Q_p$-algebra $A$ consists of finite free $A$-modules $D_A$ of rank $d$, a linear action of $I_{L/K}$ on $D_A$, and a nilpotent linear operator $N$ on $D_A$. We require that the actions of $I_{L/K}$ and $N$ commute.

Let $\mathfrak{Mod}_{\phi,N}$ be the groupoid whose fibre over a $\Q_p$-algebra $A$ consists of a module $D_A$ in $\mathfrak{Mod}_N$ equipped with a linear automorphism $\phi$ such that $l^f\phi N=N\phi$ and $\phi\gamma\phi^{-1}=\Phi\gamma\Phi^{-1}$ for all $\gamma\in I_{L/K}$. There is a natural morphism $\mathfrak{Mod}_{\phi,N}\to\mathfrak{Mod}_N$ given by forgetting $\phi$.

Given $D_A$ in $\mathfrak{Mod}_{\phi,N}$, let $\ad D_A=\Hom_A(D_A,D_A)$. Give $\ad D_A$ an operator $\phi$ by $\phi(f):=\phi\circ f\circ\phi^{-1}$, and an operator $N$ given by $N(f):=N\circ f-f\circ N$. These satisfy $l^f\phi N=N\phi$. Give $\ad D_A$ an action of $I_{L/K}$ with $\gamma\in I_{L/K}$ taking $f$ to $\gamma\circ f\circ\gamma^{-1}$. We have an anti-commutative diagram $$
\begin{CD}
	(\ad D_A)^{I_{L/K}} @>1-\phi>>(\ad D_A)^{I_{L/K}}\\
	@VVNV @VVNV \\
	(\ad D_A)^{I_{L/K}} @>l^f\phi-1>>(\ad D_A)^{I_{L/K}}
\end{CD}
$$ Let $C^\bullet(D_A)$ denote the total complex of this double complex, and $H^\bullet(D_A)$ denote the cohomology of $C^\bullet(D_A)$. 
\begin{lemma}
	\label{2.1}Let $A$ be a local $\Q_p$-algebra with maximal ideal $\mathfrak{m}_A$, and $I\subset A$ an ideal with $I\mathfrak{m}_A=0$. Let $D_{A/I}$ be in $\mathfrak{Mod}_{\phi,N}(A/I)$ and set $D_{A/\mathfrak{m}_A}=D_{A/I}\otimes_{A/I}A/\mathfrak{m}_A$.
	\begin{enumerate}
		\item If $H^2(D_{A/\mathfrak{m}_A})=0$ then there exists a module $D_A$ in $\mathfrak{Mod}_{\phi,N}(A)$ whose reduction modulo $I$ is isomorphic to $D_{A/I}$. 
		\item The set of isomorphism classes of liftings of
                  $D_{A/I}$ to $D_A$ in $\mathfrak{Mod}_{\phi,N}(A)$
                  is either empty or a torsor under
                  $H^1(D_{A/\mathfrak{m}_A})\otimes_{A/\mathfrak{m}_A}I$. Here
                  two liftings $D_A$, $D'_A$ are isomorphic if there
                  exists a map $D_A\to D'_A$ which is compatible with
                  the $\phi$-, $I_{L/K}$-, $N$-actions and which reduces to the
                  identity modulo $I$.
	\end{enumerate}
\end{lemma}
\begin{proof}
	This is very similar to the proof of Proposition 3.1.2. of
        \cite{kis06}. Let $D_A$ be a free $A$-module equipped with an
        isomorphism $D_A\otimes_A A /I\isoto D_{A/I}$. We wish to lift
        the action of $I_{L/K}$ to $D_A$ and the operator $\phi$ to an
        operator $\widetilde{\phi}$ on $D_A$ satisfying the same
        relations; equivalently, we wish to lift the corresponding
        representation of $W_K/I_L$. The obstruction to doing this is
        in $H^2(W_K/I_L,\ad D_{A/\mathfrak{m}_A})\otimes I$, which
        vanishes (this follows from the Hochschild-Serre spectral
        sequence and the fact that the group $I_{L/K}$ is finite). Similarly, $I_{L/K}$-invariant maps in $\Hom_{A/I}(D_{A/I},D_{A/I})$ lift to $I_{L/K}$-invariant maps in $\Hom_A(D_A,D_A)$, so we can lift $N$ to an endomorphism $\widetilde{N}$ which commutes with the $I_{L/K}$-action.
	
	Then $D_{A/I}$ lifts to an object of $\mathfrak{Mod}_{\phi,N}(A)$ if and only if we can choose $\widetilde{N}$, $\widetilde{\phi}$ so that $h:=\widetilde{N}-l^f\widetilde{\phi}\widetilde{N}\widetilde{\phi}^{-1}=0$. We see that $h$ induces an $I_K$-invariant map $h:D_{A/\mathfrak{m}_A}\to D_{A/\mathfrak{m}_A}\otimes_A I$, so if $H^2(D_{A/\mathfrak{m}_A})=0$ we can write $h=N(f)+(l^f\phi-1)(g)$, with $f$, $g\in(\ad D_{A/\mathfrak{m}_A})^{I_{L/K}}\otimes_{A/\mathfrak{m}_A}I$. If we then set $\widetilde{N}'=\widetilde{N}+g$, $\widetilde{\phi}'=\widetilde{\phi}-f\circ\phi$, we see that $\widetilde{N}'\widetilde{\phi}'=l^f\widetilde{\phi}'\widetilde{N}'$, which proves (1). The proof of (2) is then formally identical to the proof of part (2) of Proposition 3.1.2. of \cite{kis06} (after replacing every occurence of $p$ with $l^f$, and $G_{L/K}$ with $I_{L/K}$).
\end{proof}
\begin{corollary}
	\label{2.2}Let $A$ be a $\Q_p$-algebra and $D_A$ an object of $\mathfrak{Mod}_{\phi,N}(A)$. Suppose that the morphism $A\to\mathfrak{Mod}_{\phi,N}$ given by $D_{A}$ is formally smooth and that $H^2(D_A)=0$. Then $A$ is formally smooth over $\Q_p$.
\end{corollary}
\begin{proof}
	Identical to that of Corollary 3.1.3. of \cite{kis06}.
\end{proof}

Let $E/\Q_p$ be a finite extension, and let $D_E$ be a finite free $E$-module of rank $d$ with an action of $I_{L/K}$. Consider the functor which to an $E$-algebra $A$ assigns the set of pairs $(\phi,N)$ where $\phi$, $N$ are endomorphisms of $D_A:=D_E\otimes_E A$ making $D_A$ into an object of $\mathfrak{Mod}_{\phi,N}(A)$. This functor is represented by an $E$-algebra $B_{\phi,N}$, with $X_{\phi,N}=\Spec B_{\phi,N}$ a locally closed subscheme of $\Hom_E(D_E,D_E)^2$. Similarly, let $B_N$ represent the analogous functor considering only objects of $\mathfrak{Mod}_N(A)$, and let $X_N=\Spec B_N$. There is a natural map $X_{\phi,N}\to X_N$.
\begin{lemma}
	\label{2.3}Let $D_{B_{\phi,N}}=D_E\otimes_{\Q_p}B_{\phi,N}$, a vector bundle on $X_{\phi,N}$.
	\begin{enumerate}
		\item The morphism of groupoids on $E$-algebras $B_{\phi,N}\to\mathfrak{Mod}_{\phi,N}$ is formally smooth.
		\item There is a dense open subset $U\subset X_{\phi,N}$ such that $H^2(D_{B_{\phi,N}})|_U=0$.
	\end{enumerate}
\end{lemma}
\begin{proof}
	This is very similar to the proof of Lemma 3.1.5. of \cite{kis06}. (1) is immediate by definition. Let $U\subset X_{\phi,N}$ be the complement of the support of $H^2(D_{B_{\phi,N}})$. We must show that $U$ is dense. It suffices to show that $U$ is dense in every fibre of $X_{\phi,N}\to X_N$. Let $y\in X_N$, and let $D_y$ be the pullback to $y$ of the tautological vector bundle on $X_N$. Then $(X_{\phi,N})_y$ can be identified with an open subset (given by the non-vanishing of the determinant) of the $\kappa(y)$-vector space of linear maps $\phi:D_{y}\to D_{y}$ which satisfy the required commutation relations with $N$ and $I_{L/K}$. Then $(X_{\phi,N})_y$ is smooth and connected, so irreducible, so we need only check that if $(X_{\phi,N})_y$ is non-empty, then there is a point of $(X_{\phi,N})_y$ at which $H^2(D_{B_{\phi,N}})=0$.
	
	We have a decomposition $$D_y\isoto\oplus_\tau\Hom_{I_{L/K}}(\tau,D_y)\otimes_{\Q_p}\tau$$with $\tau$ running over the irreducible representations of $I_{L/K}$ on $\Q_p$-vector spaces. The $N$-action on $D_y$ induces an action on $D_{y,\tau}:=\Hom_{I_{L/K}}(\tau,D_y)$, and the above decomposition is then compatible with $N$-actions.
	
	For any such $\tau$ and $i$ an integer, we define
        $\tau^{\phi^i}$, with the same underlying $\Qp$-vector space
        as $\tau$, by letting $\gamma\in I_{L/K}$ act on
        $\tau^{\phi^i}$ as $\Phi^i\gamma\Phi^{-i}$ acts on
        $\tau$. Since $(X_{\phi,N})_{y}$ is nonempty, there exists a
        map $\phi:D_{y}\to D_{y}$ such that $l^{f}\phi N=N\phi$. Then
        $\phi$ induces an isomorphism $D_{y,\tau}\to
        D_{y,\tau^\phi}$. Let $n_\tau$ be the least positive integer
        such that $\tau^{\phi^{n_\tau}}=\tau$. Then since $l^f\phi
        N=N\phi$, we see that it is possible to write
        $D_{y,\tau^{\phi^i}}=\oplus_j D_{i,j}$ where $D_{i,j}$ is
        $\kappa(y)^{s_j}$ for some $s_{j}$ with canonical basis
        $(e^j_{i,1},\dots,e^j_{i,s_j})$, and
        $N(e^j_{i,k})=e^j_{i,k+1}$ $(1\leq k\leq s_j-1)$,
        $N(e^j_{i,s_j})=0$. We then consider the endomorphism $\phi$
        of $\oplus_{i=1}^{n_{\tau}} D_{y,\tau^{\phi^i}}$ given by
        $\phi(e^j_{i,k})=l^{f(s_j-k)}e^j_{i+1,k}$, where we consider
        the $i$-index to be cyclic. This satisfies $l^f\phi
        N=N\phi$. Equipping $D_y$ with the resulting choice of $\phi$,
        we find that $H^2(D_y)=0$, as required.
\end{proof}
Let $A^\circ$ be a complete Noetherian local $\bigO_E$-algebra with
finite residue field,
and $V_{A^\circ}$ be a finite free $A^\circ$-module of rank $d$
equipped with a continuous action of $G_K$. Write
$A=A^\circ[1/p]$. Define a \emph{Galois type} $\tau$ to be the
restriction to $I_K$ of a $d$-dimensional $p$-adic Weil-Deligne
representation with open kernel, assumed from now on to contain
$I_L$. We say that a $p$-adic Galois representation is of type $\tau$
if the restriction to $I_K$ of the corresponding Weil-Deligne
representation is isomorphic to $\tau$.

There is a quotient $A^\tau$ of $A$ such that for any finite
$E$-algebra $B$, a map of $E$-algebras $x:A\to B$ factors through
$A^\tau$ if and only if $V_B=V_A\otimes_A B$ has type $\tau$ (this is
clear, because the isomorphism class of a finite-dimensional
representation of a finite group in characteristic zero is determined
by its trace). It follows easily from the fact that the cohomology of
the finite group $I_{L/K}$ with coefficients in characteristic 0
vanishes in positive degree that $A^\tau$ is a union of irreducible
components of $A$.

\begin{prop}
	\label{2.4}Let $A$ be a Noetherian $\Q_p$-algebra and $D_A$ an object of $\mathfrak{Mod}_{\phi,N}$. If $A\to\mathfrak{Mod}_{\phi,N}$ is formally smooth, then there is a dense open subset $U\subset\Spec A$ such that $U$ is formally smooth over $\Q_p$, and the support of $H^2(D_A)$ does not meet $U$.
\end{prop}
\begin{proof}
	This is formally identical to the proof of Proposition 3.1.6. of \cite{kis06}, using Corollary \ref{2.2} and Lemma \ref{2.3}.
\end{proof}

Now let $V_\F=V_{A^\circ}\otimes_{A^\circ}\F$. Let $D_{V_\F}$ be the groupoid on the category of complete local $\bigO_E$-algebras with residue field $\F$, whose fibre over such an algebra $B$ consists of the deformations of $V_\F$ to a $G_K$-representation on a finite free $B$-module $V_B$.

Let $\mathfrak{m}$ be a maximal ideal of $A^\tau$, and $E'$ its residue field. For each $i\geq 1$ the $G_K$-representation $V_{A^\circ}\otimes_{A^\circ}A^\tau/\mathfrak{m}^iA^\tau$ gives an object of $\mathfrak{Mod}_{\phi,N}(A^\tau/\mathfrak{m}^i A^\tau)$ via taking the corresponding Weil-Deligne representation. Passing to the limit gives a morphism of groupoids on $E$-algebras $\hat{A}^\tau_\mathfrak{m}\to\mathfrak{Mod}_{\phi,N}$.
\begin{prop}
	\label{2.5}Suppose that the morphism $A^\circ\to D_{V_\F}$ is formally smooth. Then the morphism $\hat{A}^\tau_\mathfrak{m}\to\mathfrak{Mod}_{\phi,N}$ of groupoids on $E$-algebras is formally smooth.
\end{prop}
\begin{proof}
	This is very similar to the proof of Proposition 3.3.1. of \cite{kis06}. Let $B$ be an $E$-algebra, $I\subset B$ an ideal with $I^2=0$ and $h:\hat{A}^\tau_\mathfrak{m}\to B/I$ a map of $E$-algebras. Let $D_{B/I}$ be the object of $\mathfrak{Mod}_{\phi,N}(B/I)$ induced by $h$, and let $D_B$ be an object in $\mathfrak{Mod}_{\phi,N}(B)$ together with an isomorphism $D_B\otimes_{B}B/I\isoto D_{B/I}$. We need to show that this is induced by a map $\hat{A}^\tau_\mathfrak{m}\to B$ lifting $h$.
	
	As in the second paragraph of the proof of Proposition 3.3.1. of \cite{kis06}, we may assume that $B$ is a Noetherian complete local $E'$-algebra with residue field $E'$.
	
	Let $\mathfrak{m}_B$ be the maximal ideal of $B$. If $i\geq 1$ then $D_B\otimes_B B/\mathfrak{m}_B^i$ is an object of $\mathfrak{Mod}_{\phi,N}(B/\mathfrak{m}_B^i)$. It is admissible, as it is a repeated extension of $D_{B/I}\otimes_{B/I}B/\mathfrak{m}_B$. Then one can associate to it a finite free $B$-module $V_{B/\mathfrak{m}_B^i}$ of rank $d$, equipped with a continuous action of $G_K$, via the functor $\underline{\widehat{WD}}_{pst}$. Because $V_{B/\mathfrak{m}_B}$ has type $\tau$, we see that $V_{B/\mathfrak{m}_B^i}$ has type $\tau$ for all $i$.
	
	The result then follows as in the final paragraph of the proof of Proposition 3.3.1. of \cite{kis06}.
\end{proof}
We are now ready to prove the main result of this section. Fix an $\F$-basis for $V_\F$, and denote by $D_{V_\F}^\square$ the groupoid on the category of complete local $\bigO_E$-algebras with residue field $\F$, whose fibre over such a $B$ is an object $V_B$ of $D_{V_\F}(B)$ together with a lifting of the given basis for $V_\F$ to a $B$-basis for $V_B$. A morphism $V_B\to V_{B'}$ in $D_{V_\F}^\square$ covering $\epsilon:B\to B'$ is a $B'$-linear $G_K$-equivariant isomorphism $V_B\otimes_B B'\isoto V_{B'}$ sending the given basis of $V_B$ to that of $V_{B'}$.

Denote by $|D_{V_\F}^\square|$ the functor which to $B$ assigns the set of isomorphism classes of $D_{V_\F}^\square(B)$, and similarly for $|D_{V_\F}|$. Then $|D_{V_\F}^\square|$ is representable by a complete local $\bigO_E$-algebra $R_{V_\F}^\square$. If $\End_{\F[{G_K}]}V_\F=\F$, then $|D_{V_\F}|$ is representable by a complete local $\bigO_E$-algebra $R_{V_\F}$.
\begin{thm}
	\label{2.6}$\Spec (R_{V_\F}^\square[1/p])^\tau$ is equi-dimensional of dimension $d^2$ and admits a formally smooth, dense open subscheme. If $\End_{\F[{G_K}]}V_\F=\F$ then the same is true of $\Spec (R_{V_\F}[1/p])^\tau$, except that it is $1$-dimensional.
\end{thm}
\begin{proof}
	We give the proof for $(R_{V_\F}^\square[1/p])^\tau$, the argument for $(R_{V_\F}[1/p])^\tau$ being very similar. Let $A= (R_{V_\F}^\square[1/p])^\tau$, and let $D_A$ denote the object of $\mathfrak{Mod}_{\phi,N}(A)$ corresponding to the universal representation over $(R^\square_{V_\F}[1/p])^\tau$. It follows from Propositions \ref{2.4} and \ref{2.5} that there is a smooth dense open subscheme $U$ of $\Spec A$ such that the support of $H^2(D_A)$ does not meet $U$.
	
	To compute the dimension of $A$, it suffices to compute the dimension of the tangent spaces at closed points of $U$. Let $x$ be a closed point of $U$ with residue field $E'$, a finite extension of $E$, and write $\mathfrak{m}$ for the corresponding maximal ideal of $A$, $V_x$ for the $G_K$-representation given by specialising the universal representation over $A$, and $D_x$ for the corresponding object of $\mathfrak{Mod}_{\phi,N}(E')$. Then the dimension of the tangent space at $\mathfrak{m}$ is, by the formula found in section 2.3.4. of \cite{kis04}, $$\dim_{E'}\Ext^1(V_x,V_x)+\dim_{E'}\ad_{E'}V_x-\dim_{E'}(\ad_{E'}V_x)^{G_K}.$$ By Lemma \ref{2.1} and the fact that $H^2(D_x)=0$ (as $x\in U$), we see that 
	\begin{align*}
		\dim_{E'}\Ext^1(V_x,V_x)&=\dim_{E'}H^1(D_x)\\
		&=\dim_{E'}H^0(D_x)\\&=\dim_{E'}(\ad_{E'}V_x)^{G_K}.
	\end{align*}
	Thus the dimension of the tangent space is $\dim_{E'}(\ad_{E'} V_x)=d^2$, as required.
\end{proof}

\section{Hilbert modular forms}\label{1}
\subsection{} Let $p>2$ be a prime. Let $F$ be a totally real field,
$S$ a finite set of places of $F$ containing the places of $F$
dividing $p$ and the infinite places, and let $\Sigma\subset S$ be the
set of finite places in $S$. Let $F(S)$ denote the maximal extension
of $F$ unramified outside of $S$, and write $G_{F,S}$ for $\Gal(F(S)/F)$.

Let $E/\Q_p$ be a finite extension with ring of integers $\bigO$ and
residue field $\F$, and let $\overline{\rho}:G_{F,S}\to\GL(V)$ be a
continuous representation on a $2$-dimensional $\F$-vector space
$V$. Assume that $\overline{\rho}|_{G_{F(\zeta_p)}}$ is absolutely
irreducible, and that $\overline{\rho}$ is modular. If $p=5$ and the
projective image of $\rhobar$ is $\PGL_2(\F_5)$, assume further that
$[F(\zeta_5):F]\neq 2$. In this section we use the improvements in
\cite{kiscdm} of the work of B\"{o}ckle and Khare-Wintenberger on
presentations of global universal deformation rings, and we prove a
result guaranteeing the existence of modular lifts of
$\overline{\rho}$ with certain local properties.

Fix a continuous character $\psi:G_{F,S}\to\bigO^{\times}$, so that the composite $\psi:G_{F,S}\to\bigO^{\times}\to\F^\times$ is equal to $\det V$. We fix algebraic closures $\overline{\Q}$ and $\overline{\Q}_l$ for all primes $l$, and embeddings $\overline{\Q}\into\overline{\Q}_l$. We regard $F$ as a subfield of $\overline{\Q}$, and in a slight abuse of notation we will also write $\psi$ for the characters $G_{F_v}\to\bigO^\times$ induced by the inclusions $G_{F_v}\into G_{F,S}$. 

Let $R_{F,S}$ denote the universal deformation $\bigO$-algebra of $V$ (which exists by our assumption that $\overline{\rho}|_{G_{F(\zeta_p)}}$ (and thus $\overline{\rho}$) is absolutely irreducible), and denote by $R^\psi_{F,S}$ the quotient corresponding to deformations of determinant $\psi$. Let $\ad^0V\subset\End_\F(V)$ be the subspace of endomorphisms of $V$ having trace zero.

For each $v\in\Sigma$ we fix a basis $\beta_v$ of $V$. Then the functor which assigns to a local Artin $\bigO$-algebra $A$ with residue field $\F$ the set of isomorphism classes of pairs $(V_A,\beta_{v,A})$ with $V_A$ a deformation of $V$ (considered as a $G_{F_v}$-representation) of determinant $\psi$, and $\beta_{v,A}$ a basis of $V_A$ lifting $\beta_v$, is representable by a complete local $\bigO$-algebra $R_v^{\square,\psi}$. Let $R_{F,S}^{\square,\psi}$ be the universal deformation $\bigO$-algebra parameterising tuples $(V_A,\{\beta_{v,A}\}_{v\in\Sigma})$ with $V_A$ a deformation of $V$ considered as a $G_{F,S}$-representation, and $\beta_{v,A}$ as above. Set $R^{\square,\psi}_\Sigma=\widehat{\otimes}_{\bigO,v\in\Sigma}R_v^{\square,\psi}$. Then we have
\begin{prop}
	\label{1.1}Let $s=\sum_{v|\infty}H^0(G_{F_v},\ad^0V)$. Then for some $r\geq 0$ and $f_1,\dots,f_{r+s}\in R_{\Sigma}^{\square,\psi}[[x_1,\dots,x_{r+|\Sigma|-1}]]$ there exists an isomorphism $$R_{F,S}^{\square,\psi}\isoto R_\Sigma^{\square,\psi}[[x_1,\dots,x_{r+|\Sigma|-1}]]/(f_1,\dots,f_{r+s}).$$
\end{prop}
\begin{proof}
	This follows at once from Proposition 4.1.5. of \cite{kiscdm}, because the assumption that $\rhobar|_{G_{F(\zeta_p)}}$ is absolutely irreducible implies that $H^0(G_{F,S},(\ad^0V)^*(1))=0$.
\end{proof}
We now consider more refined deformation conditions at the places
$v\in\Sigma$; specifically, we consider deformations of specific
Galois type. Let $\epsilon$ denote the $p$-adic cyclotomic character
(which, with a slight abuse of notation, we will regard as a character
of various local or global absolute Galois groups). For each
$v\in\Sigma$ fix a Galois type $\tau_v$; that is, a representation
$\tau_v:I_{F_v}\to\GL_2(\Qpbar)$ with open kernel, that extends to a
representation of $G_{F_v}$. If $v|p$, we attach to any 2-dimensional
potentially semi-stable (so, in particular, any potentially
Barsotti-Tate) representation $\rho:G_{F_v}\to\GL_2(\Qpbar)$ a 2
dimensional $\Qpbar$-representation $WD(\rho)$ of the Weil-Deligne
group of $F_v$ (we attach this representation covariantly, as in
Appendix B of \cite{cdt}). We say that $\rho$ is potentially
Barsotti-Tate of type $\tau_v$ if it is potentially Barsotti-Tate, has
determinant a finite order character times the cyclotomic character, and
$WD(\rho)|_{I_{F_{v}}}$ is equivalent to $\tau_v$. Note that by Lemma
2.2.1.1 of \cite{bm}, any representation $\rho$ or $\tau_v$ as above
is actually defined over a finite extension of $\Qp$.

Suppose now that there is a finite order character $\chi:G_{F,S}\to\bigO^\times$ such that $\psi=\chi\epsilon$. Assume that for each $v\in\Sigma$ we have $\det\tau_v=\chi|_{I_{F_v}}$.
\begin{prop}
	\label{1.2}For each $v$ dividing $p$, there exists a unique quotient $R_v^{\square,\psi,\tau_v}$ of $R_v^{\square,\psi}$ such that: 
	\begin{enumerate}
		\item $R_v^{\square,\psi,\tau_v}$ is $p$-torsion free, reduced, and all its components are $4+[F_v:\Q_p]$-dimensional. 
		\item If $E'/E$ is a finite extension, then a map $x:R_v^{\square,\psi}\to E'$ factors through $R_v^{\square,\psi,\tau_v}$ if and only if the corresponding $E'$-representation $V_x$ is potentially Barsotti-Tate of type $\tau_v$. 
	\end{enumerate}
\end{prop}
\begin{proof}
	This follows at once from Theorem 3.3.8. of \cite{kis06}.
\end{proof}
Using the classification of irreducible components of $R_v^{\square,\psi,\tau_v}$ for the case $\tau_{v}=1$ (specifically, Corollary 2.5.16(2) of \cite{kis04}) one easily sees that each irreducible component of $R_v^{\square,\psi,\tau_v}$ corresponds either only to potentially ordinary representations or only to representations which are not potentially ordinary. Call these components ordinary, non-ordinary respectively.
\begin{prop}
	\label{1.3}For each $v$ not dividing $p$, there exists a unique quotient $R_v^{\square,\psi,\tau_v}$ of $R_v^{\square,\psi}$ such that: 
	\begin{enumerate}
		\item $R_v^{\square,\psi,\tau_v}$ is $p$-torsion free, reduced, and all its components are $4$-dimensional. 
		\item If $E'/E$ is a finite extension, then a map $x:R_v^{\square,\psi}\to E'$ factors through $R_v^{\square,\psi,\tau_v}$ if and only if the corresponding $E'$-representation $V_x$ is of type $\tau_v$. 
	\end{enumerate}
\end{prop}
\begin{proof}
	This follows from Theorem \ref{2.6}.
\end{proof}

Note that the rings $R_v^{\square,\psi,\tau_v}$ could be zero; in applications, one must check that they are nonzero (e.g. by exhibiting an appropriate deformation).

Assume now that each $R_v^{\square,\psi,\tau_v}$ is nonzero, and for
each $v\in\Sigma$ let $\overline{R}_v^{\square,\psi,\tau_v}$ denote a
quotient of $R_v^{\square,\psi,\tau_v}$ corresponding to some
irreducible component. Set
$R_\Sigma^{\square,\psi,\tau}=\widehat{\otimes}_{\bigO,v\in\Sigma}\overline{R}_v^{\square,\psi,\tau_v}$,
and
$R_{F,S}^{\square,\psi,\tau}=R_{F,S}^{\square,\psi}\otimes_{R_\Sigma^{\square,\psi}}R_\Sigma^{\square,\psi,\tau}$. Because
$\overline{\rho}$ is irreducible, there is a corresponding quotient of
$R_{F,S}^\psi$, which we denote $R_{F,S}^{\psi,\tau}$. This is the
image of $R^\psi_{F,S}$ in $R^{\square,\psi,\tau}_{F,S}$.
\begin{prop}
	\label{1.4}For some $r\geq 0$ there is an isomorphism $$R_{F,S}^{\square,\psi,\tau}\isoto R_\Sigma^{\square,\psi,\tau}[[x_1,\dots,x_{r+|\Sigma|-1}]]/(f_1,\dots,f_{r+[F:\Q]}).$$ In addition, $\dim R_{F,S}^{\square,\psi,\tau}\geq 4|\Sigma|$ and $\dim R_{F,S}^{\psi,\tau}\geq 1$. 
\end{prop}
\begin{proof}
	Note that $\rhobar$ is odd (because it is assumed modular), so that the integer $s$ in Proposition \ref{1.1} is $\sum_{v|\infty}1=[F:\Q]$. From Propositions \ref{1.2} and \ref{1.3} we have 
	\begin{align*}
		\dim R_{\Sigma}^{\square,\psi,\tau} &=\sum_{v\in\Sigma}\dim R_{v}^{\square,\psi,\tau}-(|\Sigma|-1)\\
		&=4|\Sigma|+\sum_{v|p}[F_v:\Q_p]-(|\Sigma|-1)\\&=3|\Sigma|+[F:\Q]+1.
	\end{align*}
	Thus the dimension of $R_{F,S}^{\square,\psi,\tau}$ is at least $$\dim R_{\Sigma}^{\square,\psi,\tau}+(r+|\Sigma|-1)-(r+[F:\Q])=4|\Sigma|.$$ The morphism $R_{F,S}^{\psi,\tau}\to R_{F,S}^{\square,\psi,\tau}$ is formally smooth of relative dimension $4|\Sigma|-1$, so $$\dim R_{F,S}^{\psi,\tau}=\dim R_{F,S}^{\square,\psi,\tau}-(4|\Sigma|-1)\geq 1.$$ 
\end{proof}
We now use an $R=T$ theorem to show that $R_{F,S}^{\psi,\tau}$ is finite  over $\bigO$ of $\bigO$-rank at least 1. This allows us to prove the existence of minimal liftings of specified type.
\begin{prop}
	\label{1.5}Assume that $R_{\Sigma}^{\square,\psi,\tau}\neq 0$. Assume also that the following hypothesis is satisfied: \\
	\\
	\emph{(ord)} Let $Z$ be the set of places $v$ dividing $p$ such that $\overline{R}_{v}^{\square,\psi,\tau_v}$ is ordinary. Then $\overline{\rho}\cong\overline{\rho}_{f}$ for $f$ a Hilbert modular form of parallel weight $2$, with $f$ potentially ordinary at all places in $Z$, and the corresponding automorphic representation $\pi_{f}$ not special at any place dividing $p$.\\
	\\
	Then $R_{F,S}^{\psi,\tau}$ is a finite $\bigO$-module of rank at least $1$. 
\end{prop}
\begin{proof}
	We note that it suffices to show that $R_{F,S}^{\psi,\tau}$ is
        a finite $\bigO$-algebra; if this holds, then if
        $R_{F,S}^{\psi,\tau}$ had rank $0$ as an $\bigO$-module then
        it would have dimension $0$, contradicting Proposition
        \ref{1.4} (note that $R_{F,S}^{\psi,\tau}$ is certainly
        nonzero, as it is a completed tensor product of nonzero $\mathcal{O}$-algebras).
	
	We may now (e.g. as in the proof of Theorem 3.5.5 of \cite{kis04}) choose a solvable extension $F'/F$ so that: 
	\begin{enumerate}
		\item The base change of $\pi_{f}$ to $F'$, denoted $\pi_{f_{F'}}$, is unramified or special at every finite place of $F'$.
                \item If $v'$ is a place of $F'$ lying over a place
                  $v$ in $\Sigma$, then $\tau_v|_{I_{F'_{v'}}}$ and
                  $\chi|_{G_{F'_{v'}}}$ are trivial. 
		\item $\overline{\rho}|_{G_{F'_{v'}}}$ and
                  $\chi|_{G_{F'_{v'}}}$ are trivial at all places $v'|p$ of $F'$. 
		\item $[F'(\zeta_p):F']=[F(\zeta_p):F]$, $[F':\Q]$ is even, and $\overline{\rho}|_{G_{F'(\zeta_p)}}$ is absolutely irreducible. 
		\item $\overline{\rho}|_{G_{F'}}$ is unramified at all places. 
	\end{enumerate}
	After possibly making a further base change and replacing the
        Hilbert modular form $f_{F'}$ corresponding to $\pi_{f_{F'}}$
        with a congruent form, the argument of Corollary 3.1.6 of
        \cite{kis04} (and the use of a Hecke algebra containing the
        Hecke operators at places dividing $p$) shows that we may assume that $f_{F'}$ is ordinary at precisely the places lying over places in $Z$.
	Let $S'$ denote the set of places of $F'$ lying over places in
        $S$, and $\Sigma'$ denote the set of places of $F'$ lying over
        places in $\Sigma$. Then as in the proof of Theorem 4.2.8. of
        \cite{kiscdm} (or the proof of Lemma 3.6 of \cite{kw}), $R_{F,S}$ is a finite $R_{F',S'}$-algebra.
	
	We now define a framed deformation ring which captures the
        base changes to $F'$ of the deformations parameterised by
        $R_{F,S}^{\psi,\tau}$. Restriction gives a basis $\beta_{v'}$
        of $V$ for each place $v'\in\Sigma'$. For each place
        $v'\in\Sigma'$ of $F'$ we let $R_{v'}^\square$ denote the
        universal framed deformation ring for
        $\overline{\rho}|_{G_{F_{v'}}}$. If $v'|p$, let
        $R_{v'}^{\square,0,1}$ be the quotient considered in
        \cite{kis04}; so if $E'$ is a finite extension of $E$, then a
        map $x:R_{v'}^\square\to E'$ factors through
        $R_{v'}^{\square,0,1}$ if and only if the corresponding Galois
        representation $V_x$ is Barsotti-Tate with Hodge-Tate weights
        equal to $0$, $1$, and $\det V_x$ is the cyclotomic character.
	
	It is shown in \cite{gee052} that $\Spec
        R_{v'}^{\square,0,1}[1/p]$ has precisely $2$ components, one
        corresponding to ordinary representations, and one to
        non-ordinary representations. Let
        $\overline{R}_{v'}^{\square,0,1}$ denote the quotient of
        $R_v^{\square,0,1}$ corresponding to the closure of the
        non-ordinary component if $v$ does not lie over a place in
        $Z$, and to the closure of the ordinary component if $v$ lies
        over a place in $Z$. If $v'\nmid p$ but $v'\in\Sigma'$, then
        let $\overline{R}_{v'}^{\square}$ be the quotient of
        ${R}_{v'}^{\square,\psi}$ corresponding to unramified
        representations, unless there are representations factoring
        through $\overline{R}_{v}^{\square,\psi,\tau}$ which are not
        potentially unramified, in which case let
        $\overline{R}_{v'}^{\square}$ be the quotient of
        ${R}_{v'}^{\square,\psi}$ defined in corollary 2.6.6 of
        \cite{kis04}; so if $E'$ is a finite extension of $E$, then a
        map $x:R_{v'}^\square\to E'$ factors through
        $R_{v'}^{\square}$ if and only if the corresponding Galois
        representation $V_x$ is an extension of the trivial character
        by the cyclotomic character. In either case
        $\overline{R}_{v'}^{\square}$ is a domain.
	
	Let $R_{F',S'}^\square$ denote the deformation ring which is
        universal for tuples $(V_A,\{\beta_{v',A}\}_{v'\in\Sigma'})$,
        where $V_A$ is a deformation of the $G_{F',S'}$-representation
        $\overline{\rho}|_{G_{F',S'}}$ with determinant $\psi$, which
        is unramified at all places $v'\notin\Sigma'$, together with a
        basis $\beta_{v',A}$ of $V_A$ lifting $\beta_{v'}$ for all
        $v'\in\Sigma'$. Let
        $R^\square_{\Sigma'}:=\widehat{\otimes}_{\bigO,v'\in\Sigma'}R^{\square}_{v'}$,
        $\overline{R}^\square_{\Sigma'}:=\widehat{\otimes}_{\bigO,v'\in\Sigma'}\overline{R}^\square_{v'}$. Set
        $\overline{R}^\square_{F',S'}:=R^\square_{F',S'}\otimes_{R^\square_{\Sigma'}}\overline{R}^\square_{\Sigma'}$,
        and let $\overline{R}_{F',S'}$ be the corresponding (unframed)
        deformation ring, so that $\overline{R}^\square_{F',S'}$ is
        formally smooth over $\overline{R}_{F',S'}$ of relative
        dimension $j:=4|\Sigma'|-1$.
	
Now, the proofs of Proposition 3.3.1 and Theorem 3.4.11 of \cite{kis04} show that if we identify $\overline{R}^{\square}_{F',S'}$ with a power series ring $\overline{R}_{F',S'}[[y_{1},\dots,y_{j}]]$, then $\overline{R}^{\square}_{F',S'}$ is a finite $\bigO[[y_{1},\dots,y_{j}]]$-algebra, so that $\overline{R}_{F',S'}$ is a finite $\bigO$-algebra. Since $R_{F,S}$ is a finite $R_{F',S'}$-algebra, we see that $R^{\psi,\tau}_{F,S}$ is a finite $\bigO$-algebra, as required.
\end{proof}
\begin{rem}\label{rem: ord doesn't matter if p splits}
	We note that the hypothesis (ord) is automatically satisfied
        if $F_v=\Qp$ for every $v\in Z$, by the arguments of Corollary 3.1.6 of \cite{kis04} and Lemma 2.14 of \cite{kis05}.
\end{rem}
\begin{corollary}
	\label{1.6}With the assumptions of Proposition \ref{1.5} (in
        addition to those imposed at the beginning of this section) there exists a modular lifting $\rho$ of $\overline{\rho}$ which is potentially Barsotti-Tate at all places $v|p$, which is unramified outside $S$, and which has type $\tau_v$ at all places $v\in\Sigma$. More precisely, given for each $v\in\Sigma$ a choice of component of $R_v^{\square,\psi,\tau_v}[1/p]$ satisfying \emph{(ord)}, we can choose $\rho$ so that the corresponding point of $R_v^{\square,\psi,\tau_v}[1/p]$ lies on the given component.
\end{corollary}
\begin{proof}
	The existence of such a Galois representation follows at once from Proposition \ref{1.5}, and its modularity from the proof of Proposition \ref{1.5} (note that $F'/F$ is solvable).
\end{proof}

\section{Serre weights}\label{serre}
\subsection{} In this section we use the results of section \ref{1} to prove some of results used in \cite{gee053}, and to discuss some related conjectures. We assume that $p>2$.

We begin by recalling some standard facts from the theory of quaternionic modular forms; see either \cite{taymero}, section 3 of \cite{kis04} or section 2 of \cite{kis062} for more details, and in particular the proofs of the results claimed below. We will follow Kisin's approach closely. Let $F$ be a totally real field (with no assumption on the ramification of $p$) and let $D$ be a quaternion algebra with center $F$ which is ramified at all infinite places of $F$ and at a set $\Sigma$ of finite places, which contains no places above $p$. Fix a maximal order $\bigO_D$ of $D$ and for each finite place $v\notin\Sigma$ fix an isomorphism $(\bigO_D)_v\isoto M_2(\bigO_{F_v})$. For any finite place $v$ let $\pi_{v}$ denote a uniformiser of $F_v$.

Let $U=\prod_v U_v\subset (D\otimes_F\mathbb{A}^f_F)^\times$ be a
compact open subgroup, with each $U_v\subset (\bigO_D)^\times_v$. Furthermore, assume that $U_v=(\bigO_D)_v^\times$ for all $v\in\Sigma$, and that $U_v=\GL_2(\bigO_{F_v})$ if $v|p$.

Take $A$ a topological $\Z_p$-algebra. For each $v|p$, fix a
continuous representation $\sigma_v:U_v\to\Aut(W_{\sigma_v})$ with
$W_{\sigma_v}$ a finite free $A$-module. Write
$W_\sigma=\otimes_{v|p,A}W_{\sigma_v}$ and let
$\sigma=\otimes_{v|p}\sigma_v$. We regard $\sigma$ as a representation
of $U$ in the obvious way (that is, we let $U_v$ act trivially if
$v\nmid p$). Assume that there is a character $\psi:(\mathbb{A}^f_F)^\times/F^\times\to A^\times$ such that for any place $v$ of $F$, $\sigma|_{U_v\cap\bigO_{F_v}^\times}$ is multiplication by $\psi^{-1}$. Then we can think of $W_\sigma$ as a $U(\mathbb{A}_F^f)^\times$-module by letting $(\mathbb{A}_F^f)^\times$ act via $\psi^{-1}$.

Let $S_{\sigma,\psi}(U,A)$ denote the set of continuous functions $$f:D^\times\backslash(D\otimes_F\mathbb{A}_F^f)^\times\to W_\sigma$$ such that for all $g\in (D\otimes_F\mathbb{A}_F^f)^\times$ we have $$f(gu)=\sigma(u)^{-1}f(g)\text{ for all }u\in U,$$ $$f(gz)=\psi(z)f(g)\text{ for all }z\in(\mathbb{A}_F^f)^\times.$$We can write $(D\otimes_F\mathbb{A}_F^f)^\times=\coprod_{i\in I}D^\times t_iU(\mathbb{A}_F^f)^\times$ for some finite index set $I$ and some $t_i\in(D\otimes_F\mathbb{A}_F^f)^\times$. Then we have $$S_{\sigma,\psi}(U,A)\isoto\oplus_{i\in I}W_\sigma^{(U(\mathbb{A}_F^f)^\times\cap t_i^{-1}D^\times t_i)/F^\times},$$the isomorphism being given by the direct sum of the maps $f\mapsto\{f(t_i)\}$. From now on we make the following assumption:$$\text{For all }t\in(D\otimes_F\mathbb{A}_F^f)^\times\text{ the group }(U(\mathbb{A}_F^f)^\times\cap t^{-1}D^\times t)/F^\times=1.$$ One can always replace $U$ by a subgroup (satisfying the above assumptions) for which this holds (c.f. section 3.1.1 of \cite{kis07}). Under this assumption $S_{\sigma,\psi}(U,A)$ is a finite free $A$-module, and the functor $W_\sigma\mapsto S_{\sigma,\psi}(U,A)$ is exact in $W_\sigma$.

We now define some Hecke algebras. Let $S$ be a set of finite places
containing $\Sigma$, the places dividing $p$, and the places $v$ of $F$ such that $U_v$ is not a maximal compact subgroup of $ D_v^\times$. Let $\T^{\operatorname{univ}}_{S,A}=A[T_v,S_v]_{v\notin S}$ be the commutative polynomial ring in the formal variables $T_v$, $S_v$. Consider the left action of $(D\otimes_F\mathbb{A}_F^f)^\times$ on $W_\sigma$-valued functions on $(D\otimes_F\mathbb{A}_F^f)^\times$ given by $(gf)(z)=f(zg)$. Then we make $S_{\sigma,\psi}(U,A)$ a $\T^{\operatorname{univ}}_{S,A}$-module by letting $S_v$ act via the double coset $U\bigl(
\begin{smallmatrix}
	\pi_{v}&0\\0&\pi_{v}
\end{smallmatrix}
\bigr)U$ and $T_v$ via $U\bigl(
\begin{smallmatrix}
	\pi_{v}&0\\0&1
\end{smallmatrix}
\bigr)U$. These are independent of the choices of $\pi_{v}$. We will write $\T_{\sigma,\psi}(U,A)$ or $\T_{\sigma,\psi}(U)$ for the image of $\T^{\operatorname{univ}}_{S,A}$ in $\End S_{\sigma,\psi}(U,A)$.

Let $\mathfrak{m}$ be a maximal ideal of $\T^{\operatorname{univ}}_{S,A}$. We say that $\mathfrak{m}$ is in the support of $(\sigma,\psi)$ if $S_{\sigma,\psi}(U,A)_\mathfrak{m}\neq 0$. Now let $\bigO$ be the ring of integers in  $\Qpbar$, with residue field $\F=\Fpbar$, and suppose that $A=\bigO$ in the above discussion, and that $\sigma$ has open kernel. Consider a maximal ideal $\mathfrak{m}\subset\T^{\operatorname{univ}}_{S,\bigO}$ which is induced by a maximal ideal of $\T_{\sigma,\psi}(U,\bigO)$. Then there is a semisimple Galois representation $\overline{\rho}_\mathfrak{m}:G_F\to\GL_2(\F)$ associated to $\mathfrak{m}$ which is characterised up to equivalence by the property that if $v\notin S$ and $\Frob_v$ is an arithmetic Frobenius at $v$, then the trace of $\overline{\rho}_\mathfrak{m}(\Frob_v)$ is the image of $T_v$ in $\F$. 

We are now in a position to define what it means for a representation to be modular of some weight. Let $F_v$ have ring of integers $\bigO_v$ and residue field $k_v$, and let $\sigma$ be an $\overline{\F}_p$-representation of $G:=\prod_{v|p}\GL_2(k_v)$. We also denote by $\sigma$ the representation of $\prod_{v|p}\GL_2(\bigO_{F_v})$ induced by the surjections $\bigO_{F_v}\onto k_v$.
\begin{defn}
	We say that an irreducible representation $\overline{\rho}:G_F\to\GL_2(\overline{\F}_p)$ is modular of weight $\sigma$ if for some $D$, $\sigma$, $S$, $U$, $\psi$, and $\mathfrak{m}$ as above we have $S_{\sigma,\psi}(U,\F)_\mathfrak{m}\neq 0$ and $\overline{\rho}_\mathfrak{m}\cong\overline{\rho}$.
\end{defn}

\subsection{}

It seems that the methods of \cite{gee053} do not in themselves suffice to completely prove the conjectures of \cite{bdj}. We can, however, prove the conjecture completely in the case where $p>2$ splits completely in $F$ and $\rhobar|_{G_{F(\zeta_{p})}}$ is absolutely irreducible, and we do so in section \ref{4.5}. Before doing this, we wish to discuss some conjectures which extend those of \cite{bdj}. Firstly, we discuss what we regard as the natural generalisation of their conjectures to the case where $p$ is no longer unramified in $F$. For the rest of this section, we allow $p=2$. 

Let $K/\Q_p$ be a finite extension with ring of integers $\bigO_K$ and
residue field $k_K$, and let $\rhobar_K:G_K\to\GL_2(\overline{\F}_p)$
be a continuous representation. Let the absolute ramification degree
of $K$ be $e$. If $\rho:G_{K}\to \GL_n(\Qpbar)$ is de Rham (in
particular, if it is crystalline) and $\tau:K\into \overline{\Q}_p$ we
say that the Hodge-Tate weights of $\rho$ with respect to $\tau$ are
the $i$ for which
$gr^{-i}((\rho\otimes_{\Qp}B_{dR})^{G_{K}}\otimes_{\Qpbar\otimes_{\Qp}
  K,1\otimes\tau}\Qpbar)\neq 0$. We define a set $W(\rhobar_K)$ of
irreducible representations of $\GL_2(k_K)$ in the following
fashion. Let $\sigma_{\vec{a},\vec{b}}=\otimes_{\overline{\tau}:
  k_K\into{\overline{\F}_p}}\det{}^{a_{\overline{\tau}}}
\otimes\Sym^{b_{\overline{\tau}}-1}k_K^2\otimes_{\overline{\tau}}\overline{\F}_p$
with $1\leq b_{\overline{\tau}}\leq p$. For each embedding
$\overline{\tau}:k_{K}\into\overline{\F}_p$, let $\tau^j:K\into
\overline{\Q}_p $, $j=1,\dots,e$ be the embeddings lifting
$\overline{\tau}$. Then $W(\rhobar_K)$ is the set of
$\sigma_{\vec{a},\vec{b}}$ such that there is a continuous crystalline
lift $\rho_K:G_K\to\GL_2(\overline{\Q}_p)$ of $\rho_K$ such that for
all $\overline{\tau}$ the Hodge-Tate weights of $\rho$ with respect to
$\tau^j$ are $a_{\overline{\tau}}$ and
$a_{\overline{\tau}}+b_{\overline{\tau}}$ if $j=1$ and $0$ and $1$
otherwise. Note that this definition is independent of the choice of
$\tau^1$, as for any $j$ there is an automorphism of $\Qpbar$ taking
$\tau^1$ to $\tau^j$

Let $\overline{\rho}$ be as above. Then we conjecture
\begin{conj}
  \label{3.4}$\overline{\rho}$ is modular of weight $\sigma$ if and
  only if $\sigma=\otimes_{v|p}\sigma_v$, with $\sigma_v\in
  W(\overline{\rho}|_{G_{F_v}})$ for all $v|p$.
\end{conj}
We now explain the motivation for this conjecture. Let
$\sigma=\otimes_{v|p}\sigma_v$, with each
$\sigma_v=\sigma_{\vec{a},\vec{b}}$. Suppose temporarily
that all of the $b_{\taubar}$ are odd, and that
$2a_{\taubar}+b_{\taubar}=1$ for all $\taubar$. Then we may define a
representation $\tilde{\sigma}$ of
$\prod_{v|p}\GL_2(\bigO_{F_v})$ by
$\tilde{\sigma}=\otimes\tilde{\sigma}_v$, where $\tilde{\sigma}_v=\otimes_{\overline{\tau}:
  k_v\into{\overline{\F}_p}}\det{}^{a_{\overline{\tau}}}
\otimes\Sym^{b_{\overline{\tau}}-1}F_v^2\otimes_{\tau^1}\Qpbar$
(where $k_v$ is the residue field of $F_v$). Note that the reduction
mod $p$ of $\tilde{\sigma}$ is just $\sigma$. The assumption that
$2a_{\taubar}+b_{\taubar}=1$ for all $\taubar$ allows us to lift
$\psi$ as above to a finite-order character $\tilde{\psi}$ satisfying
the required compatibility with $\tilde{\sigma}$, and Lemma \ref{432}
below, together with the Jacquet-Langlands correspondence and the
standard properties of Galois representations associated to Hilbert
modular forms, shows that $\rhobar$ is modular of weight $\sigma$ if and only
if $\rhobar$ lifts to a modular representation
$\rho:G_F\to\GL_2(\Qpbar)$ such that for all $v|p$, $\rho|_{G_{F_v}}$
is crystalline, and for all $\taubar$ the Hodge-Tate weights of $\rho$
with respect to $\tau^j$ are  $a_{\overline{\tau}}$ and
$a_{\overline{\tau}}+b_{\overline{\tau}}$ if $j=1$ and $0$ and $1$
otherwise.

Thus (under the assumption that $2a_{\taubar}+b_{\taubar}=1$ for all
$\taubar$) we see that the condition on the existence of local
crystalline lifts should be a necessary condition. On the other hand,
we conjecture that it is also sufficient, on the grounds of a general
principle that ``the only obstructions to lifting Galois
representations should be obvious'', and the Fontaine-Mazur
conjecture. Corollary \ref{1.6} is a persuasive example of the
validity of this principle. In the general case (where, for example,
some of the $b_{\taubar}$ are allowed to be even) this line of
reasoning is no longer valid, but we continue to make the
conjecture. This may be justified in two ways: firstly, it is
generally believed that there should be a purely local formulation of
the weight part of any generalisation of Serre's conjecture, and the
obstruction to reasoning as above in the general case is a global
one. Secondly, this obstruction is already present in many cases in
the original Buzzard-Diamond-Jarvis conjecture (i.e. the case where
$p$ is unramified in $F$), which has been proved in many cases in
\cite{gee053}. 

One could hope to make this conjecture more explicit by finding
conditions on $\rhobar_K$ for such a lift $\rho_K$ to exist; this is a
purely local question. One might also hope to use Corollary \ref{1.6}
to prove cases of Conjecture \ref{3.4} for general totally real
fields; we do not know whether this is possible, although we intend to
investigate the possibility in future work. Note however that in
general as the ramification of $F$ increases we expect that the set of
possible weights will also increase. The combinatorial arguments of
\cite{gee053} depend on being able to rule out many weights, because
lifts of the corresponding types do not exist, and these arguments
will not work if more weights occur.

We now wish to propose a more general conjecture that deals with
``higher'' weights. Namely, suppose that
$\sigma=\sigma_{\vec{a},\vec{b}}$ as above, except that we now drop
the assumption that $b_\tau\leq p$, so that $\sigma$ is no longer
necessarily irreducible. Exactly as above, let $K/\Q_p$ be a finite
extension with ring of integers $\bigO_K$ and residue field $k_K$, and
let $\rhobar_K:G_K\to\GL_2(\overline{\F}_p)$ be a continuous
representation. Let the absolute ramification degree of $K$ be $e$. We
define a set $W(\rhobar_K)$ of representations of $\GL_2(k_K)$ in the
following fashion. Let
$\sigma_{\vec{a},\vec{b}}=\otimes_{\overline{\tau}:
  k_K\into{\overline{\F}_p}}\det{}^{a_{\overline{\tau}}}
\otimes\Sym^{b_{\overline{\tau}}-1}k_K^2\otimes_{\overline{\tau}}\overline{\F}_p$. For each embedding
$\overline{\tau}:k_{K}\into\overline{\F}_p$, let $\tau^j:K\into
\overline{\Q}_p $, $j=1,\dots,e$ be the embeddings lifting
$\overline{\tau}$. Then $W(\rhobar_K)$ is the set of
$\sigma_{\vec{a},\vec{b}}$ such that there is a continuous crystalline
lift $\rho_K:G_K\to\GL_2(\overline{\Q}_p)$ of $\rho_K$ such that for
all $\overline{\tau}$ the Hodge-Tate weights of $\rho$ with respect to
$\tau^j$ are $a_{\overline{\tau}}$ and
$a_{\overline{\tau}}+b_{\overline{\tau}}$ if $j=1$ and $0$ and $1$
otherwise.

Let $\overline{\rho}$ be as above. Then we conjecture
\begin{conj}\label{conj:BDJ higher weight}
$\overline{\rho}$ is modular of weight $\sigma$ if and only if $\sigma=\otimes_{v|p}\sigma_{v}$, with $\sigma_v\in W(\overline{\rho}|_{G_{F_v}})$.
\end{conj}

While this formulation is striking, if one believes Conjecture
\ref{3.4} then this conjecture is equivalent to a purely local one. Indeed, we
conjecture:

\begin{conj}
  \label{conj: local formulation of higher weight BDJ}Let $K/\Qp$ be a
  finite extension with residue field $k_K$. Let
  $\rhobar_K:G_K\to\GL_2(\Fpbar)$ be a continuous representation. Let $\sigma_{\vec{a},\vec{b}}=\otimes_{\overline{\tau}:
  k_K\into{\overline{\F}_p}}\det{}^{a_{\overline{\tau}}}
\otimes\Sym^{b_{\overline{\tau}}-1}k_K^2\otimes_{\overline{\tau}}\overline{\F}_p$,
with the $b_{\taubar}\ge 1$. For each embedding
$\overline{\tau}:k_{K}\into\overline{\F}_p$, let $\tau^j:K\into
\overline{\Q}_p $, $j=1,\dots,e$ be the embeddings lifting
$\overline{\tau}$. Then the following two statements are equivalent.
\begin{enumerate}
\item $\rhobar_K$ has a lift to a continuous crystalline
  representation $\rho:G_K\to\GL_2(\Qpbar)$ such that
  \begin{itemize}
  \item  for
all $\overline{\tau}$ the Hodge-Tate weights of $\rho$ with respect to
$\tau^j$ are $a_{\overline{\tau}}$ and
$a_{\overline{\tau}}+b_{\overline{\tau}}$ if $j=1$ and $0$ and $1$
otherwise
  \end{itemize}
\item For some Jordan-H\"older factor  $\sigma'_{\vec{a'},\vec{b'}}=\otimes_{\overline{\tau}:
  k_K\into{\overline{\F}_p}}\det{}^{a'_{\overline{\tau}}}
\otimes\Sym^{b'_{\overline{\tau}}-1}k_K^2\otimes_{\overline{\tau}}\overline{\F}_p$
of $\sigma_{\vec{a},\vec{b}}$ (so that $1\le b_{\taubar}'\le p$ for
all $\taubar$), $\rhobar_K$ has a lift to a continuous crystalline
  representation $\rho':G_K\to\GL_2(\Qpbar)$ such that
  \begin{itemize}
  \item  for
all $\overline{\tau}$ the Hodge-Tate weights of $\rho$ with respect to
$\tau^j$ are $a'_{\overline{\tau}}$ and
$a'_{\overline{\tau}}+b'_{\overline{\tau}}$ if $j=1$ and $0$ and $1$
otherwise.
\end{itemize}
\end{enumerate}

\end{conj}

This conjecture can be regarded as a weak generalisation of the
Breuil-M\'{e}zard conjecture (\cite{bm}). The connection between the
two conjectures is as follows: $\rhobar$ is modular of weight
$\sigma$ if and only if $\rhobar$ is modular of weight $\sigma'$ for
$\sigma'$ some Jordan-H\"older factor of $\sigma$. Then Conjecture
\ref{conj:BDJ higher weight} follows from Conjectures \ref{3.4} and
\ref{conj: local formulation of higher weight BDJ}, together with the
definition of $ W(\overline{\rho}|_{G_{F_v}})$.

We remark that Conjecture \ref{conj:BDJ higher weight} is true provided that $F=\Q$ and $\rhobar|_{G_{\Q(\zeta_{p})}}$ is absolutely irreducible, $\rhobar|_{G_{\Qp}}\ncong\bigl(
\begin{smallmatrix}
	\omega\chi&*\\0&\chi
\end{smallmatrix}
\bigr)$ for any character $\chi$, and if $\rhobar|_{G_{\Qp}}$ has scalar semisimplification then it is scalar. The necessity of the condition in the conjecture is clear; twisting, we may assume that $\sigma=\Sym^{k-2}\F_p^2$, $k\geq 2$. If $\rhobar$ is modular of weight $\sigma$ then it lifts to a characteristic zero modular form of weight $k$ and level prime to $p$, whose associated Galois representation is crystalline of Hodge-Tate weights $0$, $k-1$. The converse follows from the proof of the relevant cases of the Breuil-Mezard conjecture in \cite{kis062}, as one sees that if $\rhobar|_{G_{\Qp}}$ has a crystalline lift of Hodge-Tate weights $0$, $k-1$ then $\rhobar$ must be modular of weight $\sigma'$ for some irreducible subquotient $\sigma'$ of $\sigma$, so that $\rhobar$ must be modular of weight $\sigma$.

\subsection{}We now state a Serre-type conjecture on the possible weights of an $n$-dimensional mod $p$ representation of $G_\Q$. For this section we allow $p=2$. We anticipate that an entirely analogous statement should be true for any number field, but for simplicity of statement, and the lack of evidence in other cases, we restrict ourselves to the rationals. Such conjectures have also been made in \cite{adp} and \cite{her06}, but the set of weights predicted in \cite{adp} is  incomplete (although one should note that it is never claimed to be a complete list), and \cite{her06} only considers tame representations.

Suppose that $$\rhobar:G_\Q\to\GL_n(\Fpbar)$$ is continuous, odd, and irreducible. Here ``odd'' means either that $p=2$, or that $|n_+-n_-|\leq 1$, where $n_+$ (respectively $n_-$) is the number of eigenvalues equal to $1$ (respectively $-1$) of the image of a complex conjugation in $G_\Q$. It is conjectured that any such representation should arise from a Hecke eigenclass in some cohomology group. We now make this precise, following \cite{her06} (although note that our  conventions differ from the ones used there).

For any positive integer $N$ with $(N,p)=1$, we define $S_1(N)$ to be the matrices in $\GL_n^+(\Z_{(N)})$ whose first row is congruent to $(1,0,\dots,0)$ mod $N$, and let $\Gamma_1(N)$ to be the matrices in $\SL_n(\Z)$ whose first row is congruent to $(1,0,\dots,0)$ mod $N$. We obtain a Hecke algebra $\mathcal{H}_1(N)$ in the usual fashion: $\mathcal{H}_1(N)$ consists of right $\Gamma_1(N)$-invariant elements in the free abelian group of right cosets $\Gamma_1(N)s$, $s\in S_1(N)$. Thus any double coset $$\Gamma_1(N)s\Gamma_1(N)=\coprod_i\Gamma_1(N)s_i$$(a finite union) is an element of $\mathcal{H}_1(N)$ in the obvious way, and we denote it $[\Gamma_1(N)s\Gamma_1(N)]$. If $M$ is a right $S_1(N)$-module then the group cohomology modules $H^\bullet (\Gamma_1(N),M)$ have a natural linear action of $\mathcal{H}_1(N)$, determined by $$[\Gamma_1(N)s\Gamma_1(N)]m=\sum_i s_i m$$for all $s\in S_1(N)$, $m\in H^0(\Gamma_1(N),M)$. 

For any prime $l\nmid N$ and $0\leq k\leq n$ we define a Hecke operator$$T_{l,k}=\left[\Gamma_1(N) \left( 
\begin{smallmatrix}
	l & & & & & \\
	&\ddots & & & & \\
	& & l& & & \\
	& & &1 & & \\
	& & & &\ddots & \\
	& & & & & 1 
\end{smallmatrix}
\right)\Gamma_1(N)\right]\in\mathcal{H}_0(N)$$(where there are $k$ $l$'s on the diagonal). Now let $M$ be a right $\Fpbar[S_1(N)]$-module, and let $\alpha\in H^0(\Gamma_1(N),M)$ be a common eigenvector of the $T_{l,k}$ for all $l\nmid Np$, $0\leq k\leq n$. We say that $\rhobar$ is \emph{attached} to $\alpha$ if for all $l\nmid Np$ we have $\rhobar$ unramified at $l$ and $$\det(1-\rhobar(\Frob_l)X)=\sum_{i=0}^n(-1)^il^{i(i-1)/2}a(l,i)X^i$$ where $a(l,i)$ is the eigenvalue of $T_{l,i}$ acting on $\alpha$, and $\Frob_l$ is an arithmetic Frobenius at $l$.

Let $F$ be a simple $\Fpbar[\GL_n(\F_p)]$-module, and let $S_1(N)$ act on $F$ via the natural projection $S_1(N)\onto\GL_n(\F_p)$. 
\begin{defn}
	We say that $F$ is a weight for $\rhobar$, or $\rhobar$ has weight $F$, if there exists an $N$ coprime to $p$ and a non-zero Hecke eigenvector $\alpha\in H^e(\Gamma_1(N),F)$ (for some $e$) such that $\rhobar$ is attached to $\alpha$.
\end{defn}

We now give a conjectural description of the set of weights for $\rhobar$. Recall that the simple $\Fpbar[\GL_n(\F_p)]$-modules are the $F(a_1,\dots,a_n)$, where $0\leq a_1-a_2,a_2-a_3,\dots,a_{n-1}-a_n\leq p-1$, and $F(a_1,\dots, a_n)$ is the socle of the dual Weyl module for $\GL_n(\F_p)$ of highest weight $(a_1,\dots,a_n)$ (see \cite{her06}). We have $F(a_1,\dots,a_n)\cong F(a_1',\dots,a_n')$ if and only if $(a_1,\dots,a_n)-(a'_1,\dots,a'_n)\in(p-1,\dots,p-1)\Z$.
\begin{conj}
	$F(a_1,\dots,a_n)$ is a weight for $\rhobar$ if and only if $\rhobar|_{G_{\Q_{p}}}$ has a crystalline lift with Hodge-Tate weights $a_1+(n-1),a_2+(n-2),\dots,a_n$.
\end{conj}

Note that this is a well-defined conjecture, as if $\rho$ is a
crystalline lift of $\rhobar|_{G_{\Q_{p}}}$ with Hodge-Tate weights
$a_1+(n-1),\dots,a_n$, then for any $k$, $\rho\otimes\epsilon^{k(p-1)}$ is also a crystalline lift, with Hodge-Tate weights $a_1+(n-1)+(p-1)k,\dots,a_n+(p-1)k$.

This conjecture is the natural generalisation of Conjecture \ref{3.4} to higher-dimensional representations. It is rather less explicit than the conjecture of \cite{her06}, but it has the advantage of applying to all $\rhobar$, rather than just the $\rhobar$ which are tame at $p$, and it also applies to all weights, rather than just the regular weights. While it is less explicit, one may well not need a more explicit formulation in applications, for example to modularity lifting theorems.

It is natural to ask whether our conjecture agrees with that of \cite{her06} in the cases where $\rhobar|_{G_{\Q_{p}}}$ is tamely ramified. We do not know if this is the case in general; one difficulty in checking this is that typically the conjectural crystalline representations are not in the Fontaine-Laffaille range. However, one can in many cases show that we predict at least as many weights as \cite{her06}. To keep the notation reasonably clear, we restrict ourselves to a single example.

Suppose that $p\geq 5$ and $\rhobar|_{G_{\Q_{p}}}=1\oplus\omega^2\oplus\omega^4$. Then \cite{her06} predicts $9$ weights for $\rhobar$, namely 

\begin{tabular}
	{llll}$(a_1,a_2,a_3)$=&$(2,1,0),$& $(p-1,p-2,4),$ &$(p-3,3,2),$\\&$(p-1,3,0)$&$(p+1,p-2,2),$&$(2p-4,p,4),$\\&$(p+2,p-2,1)$&$(2p-3,p,3),$&$(2p-1,p+2,p-2)$
\end{tabular}

\medskip(see Proposition 7.3 and Lemma 7.4 of \cite{her06}, and note that these weights were also predicted by \cite{adp}). For each of these weights we will write down a crystalline lift of $\rhobar|_{G_{\Q_{p}}}$ with the appropriate Hodge-Tate weights.

Note that e.g. from Theorem 3.2.1 of \cite{ber05} there is a $2$-dimensional crystalline representation $V$ of $G_{\Q_p}$ with Hodge-Tate weights $0$, $p+3$ and $\overline{V}=\omega\oplus\omega^3$, and a crystalline representation $W$ with Hodge-Tate weights $0$, $2p-4$ and $\overline{W}=\omega^{-3}\oplus\omega$ (in the notation of \cite{ber05}, take $V=V_{p+4,p/3}$ and $W=V_{2p-3,-p/4}$). Then the appropriate lifts are as follows: \medskip
\begin{tabular}
	{lll}&\qquad $(2,1,0)$&$\epsilon^4\oplus\epsilon^2\oplus 1$\\
	&\qquad$(p-1,p-2,4)$&$\epsilon^{p+1}\oplus\epsilon^{p-1}\oplus\epsilon^4$\\
	&\qquad$(p-3,3,2)$&$\epsilon^{p-1}\oplus\epsilon^4\oplus\epsilon^2$\\
	&\qquad$(p-1,3,0)$&$\epsilon^{p+1}\oplus\epsilon^4\oplus 1$\\
	&\qquad$(p+1,p-2,2)$&$\epsilon^{p+3}\oplus\epsilon^{p-1}\oplus\epsilon^2$\\
	&\qquad$(2p-4,p,4)$&$\epsilon^{2p-2}\oplus\epsilon^{p+1}\oplus\epsilon^4$\\
	&\qquad$(p+2,p-2,1)$ & $\epsilon V\oplus\epsilon^{p-1}$\\
	&\qquad$(2p-3,p,3)$ & $\epsilon^3W\oplus\epsilon^{p+1}$\\
	&\qquad$(2p-1,p+2,p-2)$ & $\epsilon^{p-2}V\oplus\epsilon^{p+3}$ 
\end{tabular}
\medskip
These lifts are all reducible; it seems likely that one can always produce appropriate reducible lifts whenever the 3-dimensional representation $\rhobar|_{G_{\Q_{p}}}$ is semisimple of niveau $1$, but it will not be possible in general, of course. 
 
Finally, we remark that it may be possible to make a conjecture for ``higher weights'', as we did for $\GL_2$ (see Conjecture \ref{conj:BDJ higher weight}). We do not do this for the simple reason that we have no evidence for such a conjecture. In addition, one can make an analogous conjecture for unitary groups, which one might hope to prove via the techniques of \cite{gee053} and section \ref{unitary} of this paper. We hope to return to this in future work.

\subsection{}\label{4.5}From now on we suppose that $p>2$ splits completely in $F$.
% We now consider the quaternionic forms corresponding to classical Hilbert modular forms of parallel weight $2$. Firstly, we recall the refinement of the local Langlands correspondence for $\GL_2$ given in Henniart's appendix to \cite{bm}. Suppose that $K/\Q_p$ is a finite extension, and that $\tau:I_K\to GL_2(E)$ is an inertial type, where $E$ is a finite extension of $\Q_p$. Then Henniart shows that there is a unique irreducible finite dimensional $\overline{\Q}_p$-representation $\sigma(\tau)$ of $GL_2(\bigO_K)$ with open kernel, such that if $\tilde{\tau}$ is any extension of $\tau$ to a representation of $WD_K$, and $\pi$ is the smooth representation of $\GL_2(K)$ asociated to $\tilde{\tau}$ by the local Langlands correspondence, then $\pi|_{\GL_2(\bigO_K)}$ contains $\sigma(\tau)$. Increasing $E$ if necessary, we may assume that $\sigma(\tau)$ is defined over $E$. As $\sigma(\tau)$ is a finite dimensional representation of $\GL_2(\bigO_K)$, which is compact, it contains a $\GL_2(\bigO_K)$-stable $\bigO_E$-lattice $L_{\tau}$.

% For each place $v|p$ we now choose an integer $k_v\geq 2$ and an inertial type $\tau_v:I_{F_v}\to\GL_2(E)$, and let the representation $W_{\sigma_v}$ be given by $L_{\tau_v}$. Then the space of quaternionic forms $S_{\sigma,\psi}(U,\bigO_E)$ corresponds to Hilbert modular forms of parallel weight 2, with the corresponding local Galois representations being of type $\tau_v$. 

 We recall some group-theoretic results from section 3 of \cite{cdt}. Firstly, recall the irreducible finite-dimensional representations of $\GL_2(\F_p)$ over $\overline{\Q}_p$. Once one fixes an embedding $\F_{p^2}\into M_2(\F_p)$, any such representation is equivalent to one in the following list: 
\begin{itemize}
	\item For any character $\chi:\F_p^\times\to\overline{\Q}_p^\times$, the representation $\chi\circ\det$. 
	\item For any $\chi:\F_p^\times\to\overline{\Q}_p^\times$, the representation $\operatorname{sp}_\chi=\operatorname{sp}\otimes(\chi\circ\det)$, where $\operatorname{sp}$ is the representation of $\GL_2(\F_p)$ on the space of functions $\mathbb{P}^1(\F_p)\to\overline{\Q}_p$ with average value zero. 
	\item For any pair $\chi_1\neq\chi_2:\F_p^\times\to\overline{\Q}_p^\times$, the representation $$I(\chi_1,\chi_2)=\Ind_{B(\F_p)}^{\GL_2(\F_p)}\chi_1\otimes\chi_2,$$ where $B(\F_p)$ is the Borel subgroup of upper-triangular matrices in $\GL_2(\F_p)$, and $\chi_1\otimes\chi_2$ is the character $$ \left( 
	\begin{array}{cc}
		a & b \\
		0 & d \\
	\end{array}
	\right)\mapsto\chi_1(a)\chi_2(d). $$ 
	\item For any character $\chi:\F_{p^2}^\times\to\overline{\Q}_p^\times$ with $\chi\neq\chi^p$, the cuspidal representation $\Theta(\chi)$, characterised by $$\Theta(\chi)\otimes\operatorname{sp}\cong\Ind_{\F_{p^2}^\times}^{\GL_2(\F_p)}\chi.$$
\end{itemize}
We now recall the reductions mod $p$ of these representations. Let $\sigma_{m,n}$ be the irreducible $\overline{\F}_p$-representation $\det^m\otimes\Sym^{n}\Fpbar^{2}$, with $0\leq m<p-1$, $0\leq n\leq p-1$. Then we have: 
\begin{lemma}
	\label{441}Let $L$ be a finite free $\bigO$-module with an action of $\GL_2(\F_p)$ such that $V=L\otimes_\bigO\overline{\Q}_p$ is irreducible. Let $\tilde{}$ denote the Teichm\"{u}ller lift. 
	\begin{enumerate}
		\item If $V\cong\chi\circ\det$ with $\chi(a)=\tilde{a}^m$, then $L\otimes_\bigO \F\cong\sigma_{m,0}$. 
		\item If $V\cong\operatorname{sp}_{\chi}$ with $\chi(a)=\tilde{a}^m$, then $L\otimes_\bigO \F\cong\sigma_{m,p-1}$. 
		\item If $V\cong I(\chi_1,\chi_2)$ with $\chi_i(a)=\tilde{a}^{m_i}$ for distinct $m_i\in\Z/(p-1)\Z$, then $L\otimes_\bigO \F$ has two Jordan-H\"{o}lder subquotients: $\sigma_{m_2,\{m_1-m_2\}}$ and $\sigma_{m_1,\{m_2-m_1\}}$ where $0<\{m\}<p-1$ and $\{m\}\equiv m \text{ mod }p-1$. 
		\item If $V\cong\Theta(\chi)$ with $\chi(c)=\tilde{c}^{i+(p+1)j}$ where $1\leq i\leq p$ and $j\in\Z/(p-1)\Z$, then $L\otimes_\bigO \F$ has two Jordan-H\"{o}lder subquotients: $\sigma_{1+j,i-2}$ and $\sigma_{i+j,p-1-i}$. Both occur unless $i=p$ (when only the first occurs), or $i=1$ (when only the second one occurs), and in either of these cases $L\otimes_\bigO \F\cong\sigma_{1+j,p-2}$ . 
	\end{enumerate}
\end{lemma}
\begin{proof}
	This is Lemma 3.1.1 of \cite{cdt}.
\end{proof}
In the below, we will sometimes consider the above representations as representations of $\GL_{2}(\Z_{p})$ via the natural projection map.

We now show how one can gain information about the weights associated to a particular Galois representation by considering lifts to characteristic zero. The key is the following lemma.
\begin{lemma}
	\label{432}Let $\psi:F^{\times}\backslash(\A_{F})^{\times}\to\bigO^{\times}$ be a continuous character, and write $\overline{\psi}$ for the composite of $\psi$ with the projection $\bigO^{\times}\to \F^{\times}$. Fix a representation $\sigma$ on a finite free $\bigO$-module $W_{\sigma}$, and an irreducible representation $\sigma'$ on a finite free $\F$-module $W_{\sigma'}$. Suppose that for each $v|p$ we have $\sigma|_{U_v\cap\bigO_{F_v}^\times}=\psi^{-1}|_{U_v\cap\bigO_{F_v}^\times}$ and $\sigma'|_{U_v\cap\bigO_{F_v}^\times}=\overline{\psi}^{-1}|_{U_v\cap\bigO_{F_v}^\times}$
	
	Let $\mathfrak{m}$ be a maximal ideal of $\mathbb{T}_{S,\bigO}^{\textrm{univ}}$.
	
	Suppose that $W_{\sigma'}$ occurs as a $\prod_{v|p}U_v$-module subquotient of ${W}_{\overline{\sigma}}:=W_\sigma\otimes\F$. If $\mathfrak{m}$ is in the support of $(\sigma',\overline{\psi})$, then $\mathfrak{m}$ is in the support of $(\sigma,\psi)$.
	
	Conversely, if $\mathfrak{m}$ is in the support of $(\sigma,\psi)$, then $\mathfrak{m}$ is in the support of $(\sigma',\overline{\psi})$ for some irreducible $\prod_{v|p}U_v$-module subquotient $W_{\sigma'}$ of ${W}_{\overline{\sigma}}$.
\end{lemma}
\begin{proof}
	The first part is proved just as in Lemma 3.1.4 of \cite{kis04}, and the second part follows from Proposition 1.2.3 of \cite{as862}.
\end{proof}

We now need a very special case (the tame case) of the inertial local Langlands correspondence of Henniart (see the appendix to \cite{bm}). If $\chi_{1}\neq\chi_{2}: \F_{p}^{\times}\to\bigO^{\times}$, let $\tau_{\chi_{1},\chi_{2}}$ be the inertial type $\chi_{1}\oplus\chi_{2}$ (considered as a representation of $I_{\Qp}$ via local class field theory). Then we let $\sigma(\tau_{\chi_{1},\chi_{2}})$ be a representation on a finite $\bigO$-module given by taking a lattice in $I(\chi_{1},\chi_{2})$. If $\chi:\F_{p}^{\times}\to\bigO^{\times}$, we let $\tau_{\chi,\chi}=\chi\oplus\chi$, and $\sigma(\tau_{\chi,\chi})$ be $\chi\circ\det$. If $\chi:\F_{p^{2}}^{\times}\to\bigO^{\times}$ with $\chi\neq\chi^{p}$, let $\tau_{\chi}$ be the inertial type $\chi\oplus\chi^{p}$ (considered as a representation of $I_{\Qp}=I_{\Q_{p^{2}}}$ via local class field theory). Then  we let $\sigma(\tau_{\chi})$ be a representation on a finite $\bigO$-module given by taking a lattice in $\Theta(\chi)$.

The following result then follows from Lemma \ref{432}, the
Jacquet-Langlands correspondence, the compatibility of the local and
global Langlands correspondences at places dividing $p$ (see
\cite{kis06}), and standard properties of the local Langlands
correspondence (see section 4 of \cite{cdt}); recall that the
definition of ``potentially Barsotti-Tate of type $\tau$'' includes
the requirement that the determinant be a finite order character times
the cyclotomic character.

\begin{lemma}\label{lem:typesversusweights}For each $v|p$ fix a tame
  type $\tau_{v}$ (i.e. $\tau_{v}=\tau_{\chi_{1},\chi_{2}}$ or
  $\tau_{\chi}$ as above). Suppose that $\rhobar$ is irreducible and modular of weight
  $\sigma$, and that $\sigma$ is a
  $\prod_{v|p}\GL_{2}(\bigO_{v})$-module subquotient of
  $\otimes_{v|p}\sigma(\tau_{v})\otimes_{\bigO} \Fpbar$. Then
  $\rhobar$ lifts to a modular Galois representation which is
  potentially Barsotti-Tate of type $\tau_{v}$ for each $v|p$.
	
	Conversely, if $\rhobar$ lifts to a modular Galois representation which is potentially Barsotti-Tate of type $\tau_{v}$ for each $v|p$, then $\rhobar$ is modular of weight $\sigma$ for some $\otimes_{v|p}\GL_{2}(\bigO_{v})$-module subquotient $\sigma$ of $\otimes_{v|p}\sigma(\tau_{v})\otimes_{\bigO} \Fpbar$.
	
\end{lemma}
\subsection{}  We need a slight refinement of this result in some cases, to take
  account of Hecke eigenvalues at places $v|p$. Suppose firstly that
  $A=\F$ and $\sigma$ is irreducible. We can extend the action of
  $\GL_2(\bigO_{v})$ on $\sigma$ to an action of $\GL_2(F_v)\cap
  M_2(\bigO_{F_v})$; in the case that $\sigma=\Sym^{n}\Fpbar^{2}$, we
  extend in the obvious fashion, by thinking of $\sigma$
  as homogeneous polynomials of degree $n$ in two variables $X$, $Y$,
  with $\GL_2(F_v)\cap
  M_2(\bigO_{F_v})$ acting by \[
  \left(\begin{pmatrix}
    a&b\\c&d
  \end{pmatrix}F\right)(X,Y)=F(aX+cY,bX+dY),\] and in the general case
  that $\sigma=\det{}^{m}\Sym^{n}\Fpbar^{2}$ we twist the above action
  by $\chi^{m}\circ\det$, where $\chi$ is the character
  $\Qp^{\times}\to\Fpbar^{\times}$ with $\chi(p)=1$ and
  $\chi|_{\Z_{p}^{\times}}$ equal to reduction mod $p$. We denote this
  action by $\widetilde{\sigma}$. For each $v|p$ we let
  $\GL_2(F_v)\cap M_2(\bigO_{F_v})$ act on the space of functions
  $(D\otimes_F \A_F^f)^\times\to W_\sigma$ via
  $u(f)(g):=\widetilde{\sigma}(u)(f(gu))$, so that for each $g\in
  \GL_2(F_v)\cap M_2(\bigO_{F_v})$ there is a corresponding Hecke operator $[UgU]$
  acting on $S_{\sigma,\psi}(U,\F)$. Let
  $\mathbb{T}_{S^{p},\F}^{\textrm{univ}}=\mathbb{T}_{S,\F}^{\textrm{univ}}[T_{v},S_{v}]_{v|p}$,
  and extend the action of $\mathbb{T}_{S,\F}^{\textrm{univ}}$ on $S_{\sigma,\psi}(U,\F)$ to one
  of $\mathbb{T}_{S^{p},\F}^{\textrm{univ}}$ by letting $T_{v}$,
  $S_{v}$ act via the Hecke operators corresponding to $\bigl(
\begin{smallmatrix}
	\pi_{v}&0\\0&1
\end{smallmatrix}
\bigr)$ and $\bigl(
\begin{smallmatrix}
	\pi_{v}&0\\0&\pi_{v}
\end{smallmatrix}
\bigr)$ respectively. Any maximal ideal $\mathfrak{m}$ of
$\mathbb{T}_{S^{p},\F}$ induces a maximal ideal $\mathfrak{m}'$ of
$\mathbb{T}_{S,\F}$, and we write 
$\rhobar_{\mathfrak{m}}$ for $\rhobar_{\mathfrak{m}'}$.

We also need to consider the case that $A=\bigO$ and
$\sigma=\prod_{w|p}\sigma_{w}$, where for some $v|p$ we have
$\sigma_{v}=\sigma(\tau_{v})$,
$\tau_{v}=\chi_{1,v}\oplus\chi_{2,v}$. We suppose in order to simplify
the discussion that $\chi_{{1,v}}=1$; as above, the operators in the
general case are easily defined via twisting. If $\chi_{2,v}=1$ then
$\sigma(\tau_{v})$ is trivial, and we can define $T_{v}$ in the usual
fashion. Suppose now that $\chi_{2,v}\neq 1$. Then
$\sigma(\tau_{v})=\Ind_{B(\F_p)}^{\GL_{2}(\F_{p})}1\otimes\chi_{2,v}$
(the $\bigO$-valued induction). The dual of $\sigma(\tau_{v})$ is
$\sigma(\tau_{v})^{*}=\Ind_{B(\F_p)}^{\GL_{2}(\F_{p})}1\otimes\chi_{2,v}^{-1}$,
and there is a natural $\GL_{2}(\F_{p})$-equivariant pairing
$\langle,\rangle_{v}:\sigma(\tau_{v})\times\sigma(\tau_{v})^{*}\to\bigO$. Let
$L_{\sigma(\tau_{v})^{*}}$ be the underlying $\bigO$-module of
$\sigma(\tau_{v})^{*}$. Then we let $x_{v}\in
L_{\sigma(\tau_{v})^{*}}$ be the function in
$\Ind_{B(\F_p)}^{\GL_{2}(\F_{p})}1\otimes\chi_{2,v}^{-1}$ supported on
$B(\F_{p})$ and sending the identity to $1\in\bigO$. Note that this
function is $\Gamma_{1}$-invariant, where $\Gamma_{1}$ is the
inverse image in $\GL_{2}(\Z_{p})$ of the subgroup of $\GL_2(\F_p)$ of
elements of the form $
\begin{pmatrix}
  a&b\\0 &1
\end{pmatrix}$. Then if $f\in
S_{\sigma,\psi}(U,\bigO)$, we can consider the function
$f_{x}:D^{\times}\backslash(D\otimes_{F}\A_{F}^{f})^{\times}\to\otimes_{w\neq
  v}\sigma_{w}$
given by $\langle f,x_{v}\rangle_{v}$. The map $f\mapsto f_{x}$ is an
injective map from $S_{\sigma,\psi}(U,\bigO)$ to the space of
functions
$D^{\times}\backslash(D\otimes_{F}\A_{F}^{f})^{\times}\to\otimes_{w\neq
  v}\sigma_{w}$ which are $\Gamma_{1,v}$-invariant, where
$\Gamma_{1,v}\subset U_{v}$ is the subgroup defined above. This latter
space has a natural action of a Hecke operator
$U_{v}=\bigl[\Gamma_{1,v}\left(\begin{smallmatrix} p&0\\0&1
\end{smallmatrix}\right) \Gamma_{1,v}\bigr]$. In fact, $U_v$ acts on the
image of $S_{\sigma,\psi}(U,\bigO)$; an easy way to see this is to
note that an alternative description of the map $f\mapsto f_x$ is
provided by using Frobenius reciprocity \[\Hom_{GL_2(\F_p)}(\cdot,\Ind_{B(\F_p)}^{\GL_2(\F_p)}1\otimes\chi_{2,v})=\Hom_{B(\F_p)}(\cdot,1\otimes\chi_{2,v})\] to identify
$S_{\sigma,\psi}(U,\bigO)$ with a space of functions to $\bigO\otimes\otimes_{w\neq
  v}\sigma_{w}$ which transform via $1\otimes\chi_{2,v}$ under the standard
Iwahori subgroup of $U_v$. This is obviously a subspace of the
functions which are $\Gamma_{1,v}$-invariant, and is easily seen to be $U_v$-stable.

Assume now that $\chi_{2,v}\neq 1$. Then for some
$1\leq r\leq p-2$ we have
$\sigma(\tau_{v})\otimes_{\bigO}\F=\Ind_{B(\F_{p})}^{\GL_{2}(\F_{p})}1\otimes
\delta^{r}$, where $1\otimes\delta^{r}\left(\left(\begin{smallmatrix}
      a&b\\0&d
\end{smallmatrix}\right)\right)=d^{r}$. Then (for example, by Lemma 3.2
of \cite{as86}) there is an exact sequence $$0\to \sigma_{0,r}\to
\Ind_{B(\F_{p})}^{\GL_{2}(\F_{p})}1\otimes \delta^{r}\to
\sigma_{r,p-1-r}\to 0.$$

\begin{lemma}\label{lem:compatibility of Hecke at p}
   Let $\sigma'=\sigma_{0,r}\otimes(\otimes_{w\ne
    v}\sigma_w\otimes_\bigO \F)$. The natural map
  $S_{\sigma',\psi}(U,\F)\to S_{\sigma,\psi}(U,\F)$ is equivariant for
  the actions of $T_v$ on the left hand side and $U_v$ on the right
  hand side.
\end{lemma}
\begin{proof}Dually to the above short exact sequence, we have a short exact
sequence $$0\to \sigma_{0,p-1-r}\to
\Ind_{B(\F_{p})}^{\GL_{2}(\F_{p})}1\otimes \delta^{-r}\to
\sigma_{p-1-r,r}\to 0.$$Now, the duality pairing induces a
$\GL_{2}(\Z_{p})$-equivariant pairing
$\sigma_{0,r}\times\sigma_{p-1-r,r}\to\overline{\F}_{p}$. This pairing
can be made completely explicit (see the proof of Lemma 3.1 of
\cite{as86}), and we find that if \[F(X,Y)=\sum_{j=0}^ra_jX^{r-j}Y^j,\
G(X,Y)=\sum_{j=0}^rb_jX^{r-j}Y^j,\]\[\langle
F,G\rangle=\sum_{j=0}^r{r\choose j}^{-1}(-1)^ja_jb_{r-j}.\] Then for any homogenous polynomial $F(X,Y)=\sum_{j=0}^ra_jX^{r-j}Y^j$ of
degree $r$ and any $i\in\Z_{p}$ we have

\begin{align*}\langle\left(\begin{smallmatrix}
	p&i\\0&1
\end{smallmatrix}\right)F(X,Y),X^{r}\rangle&=\langle\sum_{j=0}^ra_j(pX)^{r-j}(iX+Y)^j,X^r\rangle\\
&=\langle a_r(iX+Y)^r,X^r\rangle\\&=(-1)^ra_r\\&=\langle\sum_{j=0}^ra_jX^{r-j}Y^j,X^r\rangle\\
&=\langle F,X^{r}\rangle,\end{align*}
\begin{align*}\langle\left(\begin{smallmatrix}
	1&0\\0&p
\end{smallmatrix}\right)F,X^{r}\rangle&=\langle\sum_{j=0}^ra_jX^{r-j}(pY)^j,X^r\rangle\\&=\langle
a_0X^r,X^r\rangle\\&= 0.\end{align*}Now (see for example the proof of
Lemma 3.1 of \cite{as86}) we can take the map $\Ind_{B(\F_{p})}^{\GL_{2}(\F_{p})}1\otimes \delta^{-r}\to
\sigma_{p-1-r,r}$ in the short exact sequence above to be given by \[\phi\mapsto\sum_{i\in\F_p}\phi\left(
\begin{pmatrix}
  1 &0\\i&1
\end{pmatrix}\right)(X-iY)^r+\phi\left(
\begin{pmatrix}
  0 &1\\1&0
\end{pmatrix}\right)Y^r,\]
so that the image of $x_{v}$ in
$\sigma_{p-1-r,r}$ is just $X^{r}$. Taking an explicit set of coset
representatives for $U_{v}$, we see that if $f\in
S_{\sigma_{0,r}\otimes\sigma'',\psi}(U,\F)$, then

\begin{align*}\langle U_{v}f(g),X^{r}\rangle &=\langle\sum_{i=0}^{p-1}f(g 		\left(\begin{smallmatrix}
			p&i\\0&1
		\end{smallmatrix}\right)),X^{r}\rangle\\
&=\langle\sum_{i=0}^{p-1}\left(\begin{smallmatrix}
	p&i\\0&1
\end{smallmatrix}\right)f(g\left(\begin{smallmatrix}
	p&i\\0&1
\end{smallmatrix}\right)),X^{r}\rangle\\
&=\langle\sum_{i=0}^{p-1}\left(\begin{smallmatrix}
	p&i\\0&1
\end{smallmatrix}\right)f(g\left(\begin{smallmatrix}
	p&i\\0&1
\end{smallmatrix}\right))+	\left(\begin{smallmatrix}
		1&0\\0&p
	\end{smallmatrix}\right)f(g	\left(\begin{smallmatrix}
			1&0\\0&p
		\end{smallmatrix}\right)),X^{r}\rangle\\
		&=\langle T_{v}f(g),X^{r}\rangle.
\end{align*}
It is easy to check that the map $f\mapsto\langle f, X^{r}\rangle$ is
injective on $S_{\sigma_{0,r}\otimes\sigma'',\psi}(U,\F)$, so we
obtain the claimed compatibility of $U_{v}$ and $T_{v}$.\end{proof}

\subsection{}  The final ingredient we need before we can deduce results on the conjectures of \cite{bdj} is a characterisation of when mod $p$ representations have Barsotti-Tate lifts of particular types. Such results are proved in \cite{sav042}, and we now recall them:
\begin{thm}
	\label{442}Suppose that $\overline{\rho}:G_{\Q_p}\to\GL_2(\Fpbar)$. Then:
	\begin{enumerate}
		\item $\rhobar$ has a potentially Barsotti-Tate lift of type $\tilde{\omega}^i\oplus\tilde{\omega}^j$  only if $\rhobar|_{I_p}\cong \bigl(
		\begin{smallmatrix}
			\omega^{1+i}&*\\0&\omega^j
		\end{smallmatrix}
		\bigr),$ $ \bigl(
		\begin{smallmatrix}
			\omega^{1+j}&*\\0&\omega^i
		\end{smallmatrix}
		\bigr)$ (with $\rhobar$ peu ramifi\'{e}e if $i=j$) or $\bigl(
		\begin{smallmatrix}
			\omega_2^{k}&0\\0&\omega_2^{pk}
		\end{smallmatrix}
		\bigr)$ where $k=1+\{j-i\}+(p+1)i$. The converse holds if $\rhobar$ is decomposable or has only scalar endomorphisms.
		\item $\rhobar$ has a potentially Barsotti-Tate lift of type $\tilde{\omega_2}^m\oplus\tilde{\omega_2}^{pm}$, $(p+1)\nmid m$, only if $\rhobar|_{I_p}\cong $ $\bigl(
		\begin{smallmatrix}
			\omega_2^{1+m}&0\\0&\omega_2^{p(1+m)}
		\end{smallmatrix}
		\bigr)$, $\bigl(
		\begin{smallmatrix}
			\omega_2^{p+m}&0\\0&\omega_2^{1+pm}
		\end{smallmatrix}
		\bigr)$, $\bigl(
		\begin{smallmatrix}
			\omega^{i+j}&*\\0&\omega^{1+j}
		\end{smallmatrix}
		\bigr),$ (where $\rhobar$ is peu ramifi\'{e}e if $i=2$) or $ \bigl(
		\begin{smallmatrix}
			\omega^{1+j}&*\\0&\omega^{i+j}
		\end{smallmatrix}
		\bigr)$ (where $\rhobar$ is peu ramifi\'{e}e if $i=p-1$). Here $m=i+(p+1)j$, $1\leq i\leq p$, $j\in\Z/(p-1)\Z$. The converse holds if $\rhobar$ has only scalar endomorphisms.
	\end{enumerate}
\end{thm}
\begin{proof}
	If $\rhobar$ has only scalar endomorphisms, then this follows from Theorems 1.3 and 1.4 of \cite{sav042} (apart from the case where the type is a twist of the trivial type, which is standard). In the cases where $\rhobar$ has non-trivial endomorphisms the proof is an easy consequence of the methods of \cite{sav042}; in the ``only if'' direction, the result follows from Theorems 6.11 and 6.12 of \cite{sav042}. In the ``if'' direction, one may simply consider decomposable lifts.
\end{proof}

We now state the precise form of the conjecture for totally real fields in which $p$ splits completely. Let $\rhobar:G_{\Q_p}\to\GL_2(\F_p)$. 
\begin{defn}\label{defn:statement of weight conjecture for GL2 Qp}We define a set of weights $W(\rhobar)$ as follows.	Suppose that $\rhobar$ is irreducible. Then $\sigma_{m,n}\in W(\rhobar)$ if and only if $\rhobar|_{I_p}\cong\bigl(
	\begin{smallmatrix}
		\omega_2^{n+1}&0\\0&\omega_2^{p(n+1)}
	\end{smallmatrix}
	\bigr)\otimes\omega^m$. If $\rhobar$ is reducible, $\sigma_{m,n}\in W(\rhobar)$ only if $\rhobar|_{I_p}\cong\bigl(
	\begin{smallmatrix}
		\omega^{n+1}&*\\0&1
	\end{smallmatrix}
	\bigr)\otimes\omega^m$. If this is the case then $\sigma_{m,n}\in W(\rhobar)$ unless $n=0$ and $\rhobar$ is  tr\'{e}s ramifi\'{e}e.
\end{defn}

If now $\rhobar:G_F\to\GL_2(\Fpbar)$ is modular, we define a set of weights $W(\rhobar)$ in the obvious fashion; that is, $W(\rhobar)$ consists of precisely the representations $$\sigma_{\vec{m},\vec{n}}=\otimes_{v|p,\sigma_{m_{v},n_{v}}\in W(\rhobar|_{G_{F_v}})}\sigma_{m_{v},n_{v}}$$of $\GL_2(\bigO_F/p)=\prod_{v|p}\GL_2(\F_p)$. Assume from now on that $\rhobar|_{G_{F(\zeta_{p})}}$ is irreducible. We wish to prove that $W(\rhobar)$ is precisely the set of weights for which $\rhobar$ is modular, by using Theorem \ref{442}, Lemma \ref{441}, Lemma \ref{lem:typesversusweights} and corollary \ref{1.6}. Firstly, we will prove that if $\rhobar$ is modular of some weight, then this weight is contained in $W(\rhobar)$. We do this by using Lemma \ref{441} and Lemma \ref{lem:typesversusweights} to show that if $\rhobar$ is modular of some weight $\sigma$, then $\rhobar$ must have a potentially Barsotti-Tate lift of a particular type, and then using Theorem \ref{442} to obtain conditions on $\rhobar|_{G_{F_v}}$.

Once we have proved that any weight for $\rhobar$ is contained in $W(\rhobar)$, we will use Corollary \ref{1.6} to prove the converse. In combination with theorem \ref{442} we are able to produce modular lifts of some specified types, and then Lemma \ref{441} gives a list of irreducible representations, at least one of which must be a weight for $\rhobar$. If only one of these weights is contained in $W(\rhobar)$, then $\rhobar$ must be modular of this weight. We cannot always find a unique weight in this fashion, but in the cases where we cannot we are able to make additional arguments to complete the proof.

% We remark that while this argument examines all places above $p$ simultaneously, we can in fact work one place at a time; that is, we can prove that if $\rhobar$ is modular of weight $\otimes_{v|p}\sigma_v$ then each $\sigma_v\in W(\rhobar|_{G_{F_v}})$, and having done so then the arguments sketched above show that $\rhobar$ is modular of weight $\otimes_{v|p}\sigma_v$ for any choice of $\sigma_v\in W(\rhobar|_{G_{F_v}})$. As it is easier to write the argument in this local fashion, we will do this from now on; thus our aim is to show that the information on the relationship between ``local'' weights and types given by Theorem \ref{442} and Lemma \ref{441} precisely characterises the set $W(\rhobar|_{G_{F_v}})$.

We say that $\rhobar|_{G_{F_v}}$ is modular of weight $\sigma$ if $\rhobar$ is modular of some weight $\otimes_{w|p}\sigma_w$, with $\sigma_v=\sigma$. 
\begin{lemma}
	\label{6.1}If $\rhobar|_{G_{F_v}}$ is modular of weight $\sigma=\sigma_{m,n}$, then $\sigma\in W(\rhobar|_{G_{F_v}})$.
\end{lemma}
\begin{proof}
  Suppose firstly that $n\ne p-1$. By parts (1) and (3) of Lemma \ref{441}, $\sigma_{m,n}$ is a
  Jordan-H\"older factor of
  $\sigma(\tau_{\tilde{\omega}^{m+n},\tilde{\omega}^m})\otimes_\bigO\F$
  (using the notation explained before Lemma
  \ref{lem:typesversusweights}). By Lemma \ref{lem:typesversusweights}
  (which uses local-global compatibility at $v$), $\rhobar$ has a lift
  to characteristic zero, which when restricted to $G_{F_v}$ is
  potentially Barsotti-Tate of
  type $\tilde{\omega}^{m+n}\oplus \tilde{\omega}^{m}$. By
  part (1) of Theorem \ref{442}, we have $$\rhobar|_{I_{F_v}}\cong
  \left(
	\begin{array}{cc}
		\omega^{m+n+1} &* \\
		0& \omega^{m} 
	\end{array}
	\right),\text{ }\left( 
	\begin{array}{cc}
		\omega^{m+1} &* \\
		0& \omega^{m+n} 
	\end{array}
	\right), \text{ or }\left( 
	\begin{array}{cc}
		\omega_2^{n+1} & 0 \\
		0& \omega_2^{p(n+1)} 
	\end{array}
      \right)\otimes\omega^{m},$$and furthermore if $n=0$ and
      $\rhobar|_{G_{F_v}}$ is reducible then
      $\rhobar|_{G_{F_v}}$ is peu ramifi\'{e}e. Comparing this to
      Definition \ref{defn:statement of weight conjecture for GL2 Qp}, we are done unless
      $\rhobar|_{I_{F_v}}\cong \bigl(
	\begin{smallmatrix}
		\omega^{m+1}&*\\0&\omega^{m+n}
	\end{smallmatrix}
	\bigr)$ and $n\neq 0$, $p-1$.

        Suppose that we are in this case. By part (4) of Lemma
        \ref{441}, we see that $\sigma$ is a Jordan-H\"older factor of
        $\sigma(\tau_{\tilde{\omega}_2^{(m-1)(p+1)+n+2}})\otimes_\bigO\F$. Again,
          by Lemma \ref{lem:typesversusweights}, $\rhobar|_{G_{F_v}}$
          has a potentially Barsotti-Tate lift of type
          $\tilde{\omega}_2^{(m-1)(p+1)+n+2}\oplus\tilde{\omega}_2^{(m-1)(p+1)+p(n+2)}$. Then
          by part (2) of Theorem \ref{442}, we see that if
          $\rhobar|_{I_{F_{v}}}$ is reducible, then
          $\rhobar|_{I_{F_v}}\cong \bigl(
	\begin{smallmatrix}
		\omega^{m+n+1}&*\\0&\omega^{m}
	\end{smallmatrix}
	\bigr)$ or $\bigl(
	\begin{smallmatrix}
		\omega^{m}&*\\0&\omega^{m+n+1}
	\end{smallmatrix}
	\bigr)$. This contradiction gives the required result.

Finally, suppose $n=p-1$. Then by Lemma 3.4 of \cite{geesavitttotally}, $\rhobar|_{G_{F_v}}$
either has a potentially Barsotti-Tate lift of type
$\tilde{\omega}^m\oplus\tilde{\omega}^m$, in which case the result is
an easy corollary of Theorem 3.4.3 of \cite{MR0419467}, or $\rhobar|_{G_{F_v}}$
has a potentially semi-stable lift of type
$\tilde{\omega}^m\oplus\tilde{\omega}^m$, which is not potentially
crystalline. Any such lift is automatically ordinary, and the result
again follows easily.
\end{proof}

We now take care of (most of) the converse. We will prove a slight
refinement of the weight conjecture, where we additionally consider
Hecke operators at places dividing $p$. To this end, we have the
following slight refinement of (one direction of) Lemma
\ref{lem:typesversusweights}.

\begin{defn}
	Let $X$ be a set of places lying over $p$, and for each $v\in X$ let $\lambdabar_{v}$ be an element of $\F$. We say that an irreducible representation $\overline{\rho}:G_F\to\GL_2(\overline{\F}_p)$ is modular of weight $\sigma$ with Hecke eigenvalues $\{\lambdabar_{v}\}_{v\in X}$ if for some $D$, $\sigma$, $X$, $U$, and $\psi$ as above there is a maximal ideal $\mathfrak{m}$ of $\mathbb{T}_{S^{p},\F}^{\textrm{univ}}$ such that $(T_{v}-\lambdabar_{v})\in\mathfrak{m}$ for all $v\in X$, and we have $S_{\sigma,\psi}(U,\F)_\mathfrak{m}\neq 0$ and $\overline{\rho}_\mathfrak{m}\cong\overline{\rho}$.
\end{defn}
\begin{lemma}\label{lem:typesversusweightsrefined}Let the set of
  places of $F$ dividing $p$ be partitioned as $S=S_{1}\coprod
  S_{2}$. For each $v\in S_{1}$ fix a tame type $\tau_{v}$
  (i.e. $\tau_{v}=\tau_{\chi_{1},\chi_{2}}$ or $\tau_{\chi}$ as
  above). For each $v\in S_{2}$ fix a tame type
  $\tau_{v}=\tilde{\omega}^{m_{v}}\oplus\tilde{\omega}^{m_{v}-1}$.
	
Suppose that $\rhobar$ lifts to a modular Galois representation
  $\rho$ which is potentially Barsotti-Tate of type $\tau_{v}$ for
  each $v\in S$, and which for each $v\in S_{2}$ satisfies
	
$$\rho|_{G_{F_{v}}}\cong \left(\begin{matrix}*&*\\0&\tilde{\omega}^{m_{v}}nr(\lambda_{v})\end{matrix}\right),$$ where $nr(\lambda_{v})$ is the unramified character taking an arithmetic Frobenius element to $\lambda_{v}$. Let $\sigma_{S_{2}}=\otimes_{v\in S_{2}}\sigma_{m_{v},p-2}$. Then for some irreducible $\prod_{v\in S_{1}}\GL_{2}(\bigO_{v})$-module subquotient $\sigma_{S_{1}}$ of $\otimes_{v\in S_{1}}\sigma(\tau_{v})\otimes_{\bigO} \F$, $\rhobar$ is modular of weight $\sigma_{S_{1}}\otimes\sigma_{S_{2}}$ and Hecke eigenvalues $\{\overline{\lambda}_{v}\}_{v\in S_{2}}$.
	
      \end{lemma}
      \begin{proof}Firstly, note that by Lemma
        \ref{lem:typesversusweights}, $\rhobar$ is modular of weight
        $\sigma$ for some Jordan-H\"older factor $\sigma$ of
        $\otimes_{v|p}\sigma(\tau_{v})\otimes_{\bigO} \F$. Suppose
        $v\in S_2$. By the
        assumption on the form of $\rho|_{G_{F_v}}$, we see that \[\rhobar|_{I_{F_v}}\cong
        \begin{pmatrix}
          \omega^{m_v} &*\\0& \omega^{m_v}
        \end{pmatrix}.\] The two Jordan-H\"older factors of
        $\sigma(\tau_v)\otimes_\bigO\F$ are, by part (3) of Lemma
        \ref{441}, isomorphic to $\sigma_{m_v,p-2}$ and
        $\sigma_{m_v-1,1}$. Examining Definition \ref{defn:statement
          of weight conjecture for GL2 Qp}, we see that
        $\sigma_{m_v-1,1}\notin W(\rhobar|_{G_{F_v}})$, so by Lemma
        \ref{6.1}, we see that in fact $\rhobar$ is modular of weight
        $\sigma_{S_{1}}\otimes\sigma_{S_{2}}$ for some $\sigma_{S_1}$ as claimed.

It remains to check the claim about the Hecke eigenvalues at places in
$S_2$. This is a standard consequence of local-global compatibility
(and the calculation of the action of the Hecke algebra on a principal
series representation), together with the compatibility of the $T_v$
and $U_v$ operators proved in Lemma \ref{lem:compatibility of Hecke at p}.
\end{proof}
\begin{prop}
  \label{6.2}Suppose that $F$ is a totally real field in which the
  prime $p>2$ splits completely. Suppose that
  $\rhobar:G_F\to\GL_2(\Fpbar)$ is modular, and that
  $\rhobar|_{G_{F(\zeta_p)}}$ is irreducible. Suppose that
  $\sigma_{\vec{m},\vec{n}}=\prod_{v|p}\sigma_{v}\in W(\rhobar)$,
  where $\sigma_{v}=\sigma_{{m_{v}},n_{v}}$. Then $\rhobar$ is modular
  of weight $\sigma'_{\vec{m},\vec{n}}=\prod_{v|p}\sigma'_{v}$, where
  $\sigma'_{v}=\sigma_{v}$ unless $\sigma_{v}=\sigma_{m_{v},p-1}$ and
  $\sigma_{m_{v},0}\in W(\rhobar|_{G_{F_{v}}})$, in which case
  $\sigma_{v}'=\sigma_{m_{v},0}$.  Furthermore, if $X$ is a set of
  places $v|p$ such that $\rhobar|_{G_{F_{v}}}\cong\bigl(
		\begin{smallmatrix}
			nr(\alpha_{v})\omega^{m_{v}}&0\\0&nr(\beta_{v})\omega^{m_{v}}
		\end{smallmatrix}
		\bigr),$ then $\rhobar$ is modular of weight $\sigma'_{\vec{m},\vec{n}}$ and Hecke eigenvalues $\{\alpha_{v}\}_{{v\in X}}$. 
\end{prop}

\begin{proof}
  We make use of Corollary \ref{1.6}. Note that by Remark \ref{rem:
    ord doesn't matter if p splits} we do not have to concern
  ourselves with questions about the existence of ordinary lifts. We
  choose a tame type $\tau_{v}$ for
  each place $v|p$ such that $\rhobar|_{G_{F_{v}}}$ has a potentially
  Barsotti-Tate lift of type $\tau_{v}$. More precisely, suppose
  firstly that $\rhobar|_{G_{F_{v}}}$ is
  irreducible, and that  $n_v\ne p-1$. Then by Definition \ref{defn:statement of weight
    conjecture for GL2 Qp} we see that since $\sigma_{m_v,n_v}\in
  W(\rhobar|_{G_{F_v}})$ we have \[\rhobar|_{I_{F_v}}\cong
  \begin{pmatrix}
    \omega_2^{(n_v+1)+(p+1)m_v}&0\\0&\omega_2^{p(n_v+1)+(p+1)m_v}
  \end{pmatrix}.\]
We choose
  $\tau_{v}=\tilde{\omega}_2^{n_v+2+(p+1)(m_v-1)}\oplus\tilde{\omega}_2^{n_v+2+(p+1)(m_v-1)}$, and
  note that $\rhobar|_{G_{F_v}}$ has a potentially Barsotti-Tate lift
  of this type by part (2) of Theorem \ref{442}. Suppose secondly that
  $n_{v}=p-1$ and $\rhobar|_{I_{F_{v}}}$ is reducible and tr\'{e}s ramifi\'{e}e;
  then we choose an arbitrary type $\tau_{v}$ such that
  $\rhobar|_{G_{F_{v}}}$ has a lift of type $\tau_{v}$. Finally, in all other cases we let
  $\tau_v=\widetilde{\omega}^{m_{v}+n_{v}}\oplus\widetilde{\omega}^{m_{v}}$
  (note that since it is assumed that $\sigma_{m_v,n_v}\in
  W(\rhobar|_{G_{F_v}})$, it follows from Definition
  \ref{defn:statement of weight conjecture for GL2 Qp} and Theorem
  \ref{442} that  $\rhobar|_{G_{F_v}}$ has a potentially Barsotti-Tate lift
  of this type). If $v\in X$, we note that for some finite-order character $\psi_v$ as in example (E4) of \cite{gee051}, there
  is a component of a local framed deformation ring
  $R_v^{\square,\psi_v\epsilon,\tau_v}[1/p]$ such that the corresponding Galois
  representations are all of the form $\bigl(
		\begin{smallmatrix}
			*&*\\0&nr(\tilde{\alpha}_{v})\widetilde{\omega}^{m}
		\end{smallmatrix}
		\bigr)$ where $\tilde{\alpha}_{v}$ is a lift of
                $\alpha_{v}$. 

For each $v|p$ choose a potentially Barsotti-Tate lift $\rho_v$ of $\rhobar|_{G_{F_v}}$ of type
$\tau_v$ (lying on the specified component if $v\in X$). Now let
$\psi_v=\det\rho_v\epsilon^{-1}$, and choose a finite order character
$\psi:G_F\to\Qpbar^\times$ with $\psi|_{G_{F_v}}=\psi_v$ for all $v|p$
(this is possible by (the proof of) Lemma 4.1.1 of \cite{cht}). Let
$\Sigma$ be the set of places at which $\psi$ is ramified, together
with the set of places of $F$ dividing $p$. From Lemma \ref{432} one
sees that $\rhobar$ has a lift of determinant $\psi\epsilon$, and in
particular for each $v\in \Sigma$, $v\nmid p$, $\rhobar|_{G_{F_v}}$
has a lift $\rho_v$ of determinant $\psi|_{G_{F_v}}\epsilon$. For each
$v\in \Sigma$, choose a component of 
$R_v^{\square,\psi_v\epsilon,\tau_v}[1/p]$ containing
$\rho_v$. Applying Corollary \ref{1.6} to $\rhobar$ and the above
choices of $\psi$,
$\tau_v$, we conclude that $\rhobar$ has a modular lift
$\rho:G_F\to\GL_2(\Qpbar)$ which is potentially Barsotti-Tate of type
$\tau_v$ for all $v|p$, and which for all $v\in X$
satisfies $$\rho|_{G_{F_{v}}}\cong
\left(\begin{matrix}*&*\\0&\tilde{\omega}^{m_{v}}nr(\tilde{\alpha}_{v})\end{matrix}\right),$$ where $\tilde{\alpha}_{v}$ is a lift of
                $\alpha_{v}$.

Applying Lemma \ref{lem:typesversusweightsrefined}, we see that
$\rhobar$ is modular of some weight $\sigma'_{\vec{m},\vec{n}}=\prod_{v|p}\sigma'_{v}$ and Hecke
eigenvalues $\{\alpha_{v}\}_{{v\in X}}$, where $\sigma'_v=\sigma_v$ if
$v\in X$, and otherwise $\sigma_v$ is a Jordan-H\"older factor of
$\sigma(\tau_v)\otimes_\bigO\F$. We now examine the choices of
$\tau_v$ made above. If $\rhobar|_{G_{F_v}}$ is irreducible and
$n_v\ne p-1$ then by part (4) of Lemma \ref{441} we see that the Jordan-H\"older factors of
$\sigma(\tau_v)\otimes_\bigO \F$ are $\sigma_{m_v,n_v}$ and
$\sigma_{m_v+n_v+1,p-3-n_v}$. However, this second weight is not
contained in $W(\rhobar|_{G_{F_v}})$ by definition, so by Lemma
\ref{6.1} we have $\sigma_v'=\sigma_v$. 

If  $n_{v}=p-1$ and $\rhobar|_{I_{F_{v}}}$ is reducible and tr\'{e}s
ramifi\'{e}e, then by definition $W(\rhobar|_{G_{F_v}})=\sigma_v$, so
again by Lemma
\ref{6.1} we have $\sigma_v'=\sigma_v$. In the other cases where
$n_{v}=p-1$, and the cases where $n_v=0$, we have $\sigma(\tau_v)\otimes_\bigO\F=\sigma_{m_v,0}$
by part (1) of Lemma \ref{441}.

Finally, in the remaining cases, we have that $\rhobar|_{G_{F_v}}$ is
reducible, so since $\sigma_{m_v,n_v}\in W(\rhobar|_{G_{F_v}})$ by
assumption, we
have \[\rhobar|_{I_{F_v}}\cong
\begin{pmatrix}
  \omega^{m_v+n_v+1}&*\\0&\omega^{m_v}
\end{pmatrix}.\] In addition, by part (3) of Lemma \ref{441} the
Jordan-H\"older factors of $\sigma(\tau_v)\otimes_\bigO\F$ are
$\sigma_{m_v,n_v}$ and $\sigma_{m_v+n_v,p-1-n_v}$.  However, this second weight is not
contained in $W(\rhobar|_{G_{F_v}})$ (by definition), so by Lemma
\ref{6.1} we have $\sigma_v'=\sigma_v$. \end{proof}

 We now use an argument suggested to us by Mark Kisin. Note that
 although we are assuming throughout this section that $p$ splits
 completely in $F$, Lemma \ref{weight 2 to weight p-1} below only requires that $F_{v}=\Qp$.

Let $G=\GL_2(\Q_p)$, $K=\GL_2(\Z_p)$, and let $Z$ be the centre of $G$. If $\sigma$ is any representation of $KZ$ on a finite dimensional $\F$-vector space $V_\sigma$, then we let $\cind\sigma$ denote the compact induction of $\sigma$, and $\ind\sigma$ the induction of $\sigma$ with no restriction on supports (note that this notation is not standard). It is easy to check that if $\sigma^{*}$ denotes the dual of $\sigma$ then $\ind\sigma^{*}$ is dual to $\cind\sigma$.

We recall some of the results on Hecke algebras proved in \cite{bl}. A $KZ$-bi-invariant function $\phi:G\to\End_\F V_\sigma$ is one which satisfies $\phi(h_1gh_2)=\sigma(h_1)\phi(g)\sigma(h_2)$ for all $g\in G$, $h_1$, $h_2\in KZ$. Any such function acts on $\cind\sigma$ as follows: if $f\in\cind\sigma$, then $$\left(\phi(f)\right)(g)=\sum_{KZy\in KZ\backslash G}\phi(gy^{-1})f(y)=\sum_{yKZ\in G/KZ}\phi(y)f(y^{-1}g)$$(see Proposition 5 of \cite{bl}). Now give $\F^2$ an action of $KZ$ with $K$ acting via the natural map $\GL_2(\Z_p)\to\GL_2(\F_p)$, and $p\in Z$ acting trivially. Let $r\in[0,p-1]$, and set $\sigma=\Sym^r\F^2$. Let $T$ be the endomorphism of $\cind\sigma$ corresponding to the $KZ$-bi-invariant function which is supported on $KZ\left(
\begin{smallmatrix}
	1&0\\0&p^{-1}
\end{smallmatrix}
\right)KZ$ and which takes $\left(
\begin{smallmatrix}
	1&0\\0&p^{-1}
\end{smallmatrix}
\right)$ to $\Sym^r\left(
\begin{smallmatrix}
	0&0\\0&1
\end{smallmatrix}
\right)$. By Proposition 8 of \cite{bl}, $\F[T]$ is the full endomorphism algebra of $\cind\sigma$. One obtains a dual action on $\ind\sigma^{*}$.

For any $\lambda\in\F$, set
$\pi(r,\lambda)=\cind(\sigma)/(T-\lambda)\cind(\sigma)$. From Theorem
30 of \cite{bl} and Theorem 1.1 of \cite{bre031} we see that
$\pi(0,\lambda)$ is irreducible unless $\lambda=\pm 1$, when
$\pi(0,\pm1)$ is a non-trivial extension of $\mu_{\pm 1}\circ\det$ by
$\mu_{\pm 1}\circ\det\otimes\operatorname{Sp}$, where $\mu_{\pm
  1}:\Q_p^\times\to\F^\times$ is the irreducible character sending
$p\mapsto\pm 1$, and $\operatorname{Sp}$ is a certain irreducible
representation. Similarly, $\pi(p-1,\lambda)$ is irreducible unless $\lambda=\pm 1$, when
$\pi(p-1,\pm1)$ is a non-trivial extension of $\mu_{\pm
  1}\circ\det\otimes\operatorname{Sp}$ by $\mu_{\pm 1}\circ\det$.
\begin{lemma}
	\label{6.3} There is a morphism of
        $\F[T][\GL_2(\Q_p)]$-modules $$\theta:\ind\mathbf{1}\to\ind\Sym^{p-1}\F^2$$
        such that for all $\lambda\in \F$, the kernel of the induced
        map $$\theta:(\ind\mathbf{1})[T-\lambda]\to(\ind\Sym^{p-1}\F^2)[T-\lambda]$$  contains only functions which factor through the determinant.
\end{lemma}
\begin{proof}
	In Lemma 1.5.5 of \cite{kis062} there is a construction of a morphism of $\F[T][\GL_2(\Q_p)]$-modules $$\cind\Sym^{p-1}\F^2\to\cind\mathbf{1}$$which is nonzero modulo $T-\lambda$ for all $\lambda\in \F$, and which induces an isomorphism modulo $T-\lambda$ for all $\lambda\neq\pm 1$. From the above discussion, we see that if $\lambda=\pm 1$, then we have an induced morphism $\pi(p-1,\pm 1)\to\pi(0,\pm 1)$ whose cokernel is $\mu_{\pm 1}\circ\det$. Taking duals gives the required morphism.
\end{proof}
We can now construct a form of weight $\sigma_{m,p-1}$ from one of weight $\sigma_{m,0}$. 
\begin{lemma}\label{weight 2 to weight p-1}If $\rhobar$ is modular of weight $\sigma=\otimes_{w|p}\sigma_w$ and Hecke eigenvalues $\{\alpha_{w}\}_{w\in S}$ for some $S$ not containing $v$, and $\sigma_{v}=\sigma_{m_{v},0}$, then $\rhobar$ is also modular of weight $\sigma'=\sigma_{m_{v},p-1}\otimes_{w\neq v}\sigma_w$ and Hecke eigenvalues $\{\alpha_{w}\}_{w\in S}$.
	\end{lemma}
	\begin{proof}
By definition there is a maximal ideal $\mathfrak{m}$ of $\mathbb{T}_{S^{p},\F}^{\textrm{univ}}$ with $(T_{w}-\alpha_{w})\in \mathfrak{m}$ for all $w\in S$ and some $U$ such that $S_{\sigma,\psi}(U,\F)_\mathfrak{m}\neq 0$ and $\rhobar_\mathfrak{m}\otimes\Fpbar\cong\rhobar$.

					Suppose $f\in S_{\sigma,\psi}(U,\F)$. Then $$f:D^\times\backslash(D\otimes_F\A_F^f)^\times\to \sigma$$ is such that for all $g\in(D\otimes_F\A^f_F)^\times$ we have 
					\begin{align*}
						f(gu)=\sigma(u)^{-1}f(g)\text{ for all }u\in U\\
						f(gz)=\psi(z)f(g)\text{ for all }z\in (\A_F^f)^\times.
					\end{align*}
					Consider $f$ as a
                                        $D^{\times}$-invariant
                                        function
                                        $\tilde{f}:(D\otimes_F\A_F^f)^\times\to\sigma$. Let
                                        $\widetilde{U}=GL_2(F_{v})\times\prod_{x\neq
                                          v}U_{x}$. Then by Frobenius
                                        reciprocity there is a
                                        canonical
                                        isomorphism \[\Map_{U}((D\otimes_{F}\A_{F}^{f})^{\times},\sigma)\isoto\Map_{\widetilde{U}}((D\otimes_{F}\A_{F}^{f})^{\times},\otimes_{w\neq
                                          v}\sigma_w\otimes\ind\sigma),\]so
                                        we see that the map $\theta$
                                        of Lemma \ref{6.3} induces a
                                        map \[\Map_{U}((D\otimes_{F}\A_{F}^{f})^{\times},\sigma)\to\Map_{U}((D\otimes_{F}\A_{F}^{f})^{\times},\sigma').\]
                                        Furthermore, since Frobenius
                                        reciprocity is functorial with
                                        respect to the first argument,
                                        we easily see that this map is
                                        compatible with the action of
                                        the Hecke operators at all
                                        places other than $v$ for
                                        which they are defined and the
                                        action of
                                        $(\A_{F}^{f})^{\times}$, and
                                        that it takes $\tilde{f}$ to a
                                        $D^{\times}$-invariant
                                        function. Thus we have an
                                        induced
                                        map of $\F[T]$-modules \[S_{\sigma,\psi}(U,\F)_\mathfrak{m}\to
                                        S_{\sigma',\psi}(U,\F)_\mathfrak{m}.\]
                                        It remains to checked that this map is
                                        injective. Since
                                        $S_{\sigma,\psi}(U,\F)_\mathfrak{m}$
                                        is a finite $\F[T]$-module, it
                                        suffices to check that for
                                        each $\lambda\in \F$ the map  \[S_{\sigma,\psi}(U,\F)_\mathfrak{m}[T-\lambda]\to
                                        S_{\sigma',\psi}(U,\F)_\mathfrak{m}[T-\lambda]\]
                                        is injective. If $f$ is in
                                        the kernel of the map, then
                                         one sees that
                                        $\tilde{f}$ is left
                                        $\SL_{2}(F_{v})$-invariant by
                                        Lemma \ref{6.3}. Because
                                        it is also left
                                        $D^{\times}$-invariant, strong
                                        approximation shows that is in
                                        fact invariant under the
                                        subgroup of
                                        $(D\otimes_{F}\A_{F}^{f})^{\times}$
                                        of elements of reduced norm
                                        1. Thus $f$ factors through
                                        the quotient of
                                        $(D\otimes_{F}\A_{F}^{f})^{\times}$
                                        by the elements of reduced
                                        norm 1. If $f\ne 0$ this implies that
                                        $\rhobar_{\mathfrak{m}}$ is
                                        reducible, a contradiction.
	\end{proof}

Putting all this together, we have proved:
\begin{thm}
	 Suppose that $F$ is a totally real field in which the prime $p>2$ splits completely. Suppose that $\rhobar:G_F\to\GL_2(\Fpbar)$ is modular, and that $\rhobar|_{G_{F(\zeta_p)}}$ is irreducible. Then $\rhobar$ is modular of weight $\sigma$ if and only if $\sigma\in W(\rhobar)$. Furthermore, if $\sigma\in W(\rhobar)$ and $X$ is a set of places such that for any $v\in X$ we have $\rhobar|_{G_{F_{v}}}\cong\bigl(
		\begin{smallmatrix}
			nr(\alpha_{v})\omega^{m_{v}}&0\\0&nr(\beta_{v})\omega^{m_{v}}
		\end{smallmatrix}
		\bigr),$ then $\rhobar$ is modular of weight $\sigma$ and Hecke eigenvalues $\{\alpha_{v}\}_{{v\in X}}$.
\end{thm}
\begin{proof}
  This follows immediately from Proposition \ref{6.2} and Lemma
  \ref{weight 2 to weight p-1} (applied successively to the places $v$
  for which $n_v=p-1$ and $\sigma_{m_v,0}\in
  W(\rhobar|_{G_{F_v}})$).
\end{proof}

\section{Automorphic representations on unitary groups}\label{unitary}
\subsection{}We now extend some of our results on the existence of
automorphic liftings of prescribed types to the case of
$n$-dimensional representations. Unfortunately, the results we obtain
are rather weaker than those for Hilbert modular forms, for several
reasons. Firstly, one can no longer expect to work directly with
$\GL_n$, as the Taylor-Wiles method breaks down for $\GL_n$ if $n>2$
(see the introduction to \cite{cht}). One needs instead to work with
representations satisfying some self-duality conditions; we choose, as
in \cite{cht} and \cite{tay06}, to work with representations into the
disconnected group $\gn$, which we define below. Such Galois
representations are associated to automorphic representations on
unitary groups, for which the Taylor-Wiles method does work.

We choose to follow the notation of \cite{cht} and \cite{tay06} for
the most part, rather than that used in the rest of this paper. We
apologise for this, but we hope that this should make it easier for
the reader to follow the various references we make to \cite{cht},
\cite{tay06} and \cite{guerberoff2009modularity}. In particular, we work with $l$-adic, as opposed to
$p$-adic, representations. With this in mind, we can state the second
major restriction on our results. In section \ref{1} we were able
(under a hypothesis on the existence of ordinary lifts) to choose the
type of our lifting at any finite place, including those dividing the
characteristic of the residual representation. In the $n$-dimensional
case, we are limited to considering places not dividing $l$ which
split in the CM field used to define our unitary group. The reason we cannot change types
at places dividing $l$ is the absence of an appropriate $R=T$
theorem. In \cite{cht}, \cite{tay06}, and \cite{guerberoff2009modularity} results are only obtained for
representations which are crystalline over an unramified extension of
$\Q_l$, whereas to consider non-trivial types at $l$ one needs to be
able to prove an $R=T$ theorem for representations which only become
crystalline over a ramified extension. To our knowledge, the only such
theorems in the literature are those for weight two Hilbert modular
forms in \cite{kis04} and \cite{gee052}, and the work of \cite{kis062}
for modular forms over totally real fields in which $p$ splits
completely. However, the work of \cite{tay06} reduces the task of
proving such theorems to a purely local analysis of the irreducible
components of certain deformation spaces, and we hope to return to
this question in the future. Note that if we had such $R=T$ theorems
then the framework presented here would immediately allow us to prove
a theorem allowing one to choose the type of an automorphic lift at
places dividing $l$.

When the first draft of this paper was written, the only $R=T$
theorems available for unitary groups were those of \cite{cht} and
\cite{tay06}, which required a square-integrability hypothesis at a
finite place. Thanks to the proof of the fundamental lemma and
subsequent work of Chenevier, Harris, Labesse and Shin, this
hypothesis has been removed in \cite{guerberoff2009modularity}, which
has allowed us to strengthen our results while simplifying their
proofs. $R=T$ theorems have also been proved in the
ordinary case by Geraghty (\cite{ger}), and the analogues of the
results of this section are proved in \cite{GG}.

We begin by recalling from section 1.1 of \cite{cht} the definition of
the group $\gn$, and the relationship between representations valued
in $\gn$ and essentially self-dual representations valued in
$\GL_n$. Let $\gn$ be the group scheme over $\Z$ which is the
semi-direct product of $\GL_n\times\GL_1$ by the group $\{1,j\}$,
which acts on $\GL_n\times\GL_1$
via $$j(g,\mu)j^{-1}=(\mu^tg^{-1},\mu).$$ Let $\nu:\gn\to\GL_1$ be the
homomorphism sending $(g,\mu)\mapsto\mu$ and $j\mapsto -1$, and let
$\gn^0$ be the connected component of $\gn$. We let
$\mathfrak{g}_n:=\Lie\GL_n\subset\Lie\gn$, and let $\ad$ denote the
adjoint action of $\gn$ on $\mathfrak{g}_n$. Let $\mathfrak{g}_n^0$ be
the trace zero subspace of $\mathfrak{g}_n$. Let $F/F^+$ be an
extension of number fields of degree $2$, with $F^+$ totally real and
$F$ CM. Let $c\in G_{F^+}$ be a complex conjugation. Then we have
\begin{lemma}
	\label{5.1}Suppose that $k$ is a field, that $\chi:G_{F^+}\to k^\times$ is a continuous homomorphism, and that $$\rho:G_F\to\GL_n(k)$$is absolutely irreducible, continuous, and satisfies $\chi\rho^\vee\cong\rho^c$. Then there exists a continuous homomorphism $$r:G_{F^+}\to\gn(k)$$ such that $r|_{G_F}=\rho$, $(\nu\circ r)|_{G_F}=\chi|_{G_F}$, and $r(c)\notin\gn^{0}(k)$. There is a bijection between $\GL_n(k)$-conjugacy classes of such extensions of $\rho$ and $k^\times/(k^\times)^2$, so that in particular if $k$ is algebraically closed then $r$ is unique up to $GL_n(k)$-conjugacy.
\end{lemma}
\begin{proof}
	This is a special case of Lemma 1.1.4 of \cite{cht}.
\end{proof}

We say that a cuspidal automorphic representation $\pi$ of $\GL_n(\A_F)$ is RACSDC (regular, algebraic, conjugate self dual, cuspidal) if 
\begin{itemize}
	\item $\pi^\vee\cong\pi^c$, and
	\item $\pi_\infty$ has the same infinitesimal character as some irreducible algebraic representation of $\Res_{F/\Q}\GL_n$.
\end{itemize}
We wish to define the weight of such a representation. Let $a\in(\Z^n)^{\Hom(F,\C)}$ satisfy 
\begin{itemize}
	\item $a_{\tau,1}\geq\dots\geq a_{\tau,n}$ 
	\item $a_{\tau c,i}=-a_{\tau,n+1-i}$
\end{itemize}for all $\tau\in\Hom(F,\C)$, where $c$ again denotes complex conjugation.
Let $\Xi_a$ denote the irreducible algebraic representation of $\GL_n^{\Hom(F,\C)}$ which is the tensor product over $\tau\in\Hom(F,\C)$ of the irreducible representations of $\GL_n$ with highest weights $a_\tau$. We say that an RACSDC representation $\pi$ of $\GL_n(\A_F)$ has weight $a$ if $\pi_\infty$ has the same infinitesimal character as $\Xi_a^\vee$.

Let $\pi$ be an irreducible smooth $\Qbar_l$-representation of
$\GL_n(K)$, $K$ a finite extension of $\Q_p$, $p\ne l$. Then as in section 3.1
of \cite{cht} we let $r_l(\pi)$ denote the $l$-adic Galois
representation associated to the Weil-Deligne representation
$\rec_l(\pi^\vee\otimes |\cdot|^{(1-n)/2})$, provided this exists.
Here $\rec_l$ is the local Langlands correspondence of \cite{ht01},
so that for example if $\pi=\nind(\chi_1\times\dots\times\chi_n)$ is
an irreducible principal series representation,
then \[\rec_l(\pi)=\chi_1\circ\Art_K^{-1}\oplus\dots\oplus\chi_n\circ\Art_K^{-1}.\]
\begin{prop}
	\label{5.2}Let $\iota:\Qbar_l\isoto\C$. Suppose that $\pi$ is an RACSDC automorphic representation of $\GL_n(\A_F)$. Then there is a continuous representation $r_{l,\iota}(\pi):G_F\to\GL_n(\Qbar_l)$ satisfying:
	\begin{enumerate}
		\item If $v\nmid l$
                  then $$r_{l,\iota}(\pi)|^{ss}_{G_{F_v}}=r_l(\iota^{-1}\pi_{v})^\vee(1-n)^{ss}.$$
		\item $r_{l,\iota}(\pi)^c=r_{l,\iota}(\pi)^\vee\epsilon^{1-n}$, where $\epsilon$ is the $l$-adic cyclotomic character. 
		\item If $v|l$ then $r_{l,\iota}(\pi)|_{G_{F_v}}$ is potentially semi-stable, and if furthermore $\pi_{v}$ is unramified then $r_{l,\iota}(\pi)|_{G_{F_v}}$ is crystalline.
	\end{enumerate}
\end{prop}
\begin{proof}
	See Theorem 1.2 of \cite{BLGHT}.
\end{proof}

We can take $r_{l,\iota}(\pi)$ to be valued in $\GL_n(\bigO)$ with
$\bigO$ the ring of integers of a finite extension of $\Q_l$; reducing
modulo the maximal ideal of $\bigO$ and semisimplifying gives a well
defined semisimple
representation $$\overline{r}_{l,\iota}(\pi):G_F\to\GL_n(\Fbar_l).$$
We say that a continuous semisimple representation
$r:G_F\to\GL_n(\Qbar_l)$ is automorphic of weight $a$ if it is isomorphic to $r_{l,\iota}(\pi)$ for some
$\iota:\Qbar_l\isoto\C$, and some RACSDC automorphic form $\pi$ of
weight $a$. We say that it is
automorphic of weight $a$ and level prime
to $l$ if furthermore $\pi_l$ is unramified. We say that a continuous,
semisimple representation $\overline{r}:G_F\to\GL_n(\Fbar_l)$ is
automorphic of weight $a$ if it lifts
to a representation $r:G_F\to\GL_n(\Qbar_l)$ which is automorphic of
weight $a$ and level prime to $l$.

From Lemma \ref{5.1} and Proposition \ref{5.2}(2), we see that there
is a unique extension (up to $\GL_n(\Qbar_l)$-conjugation) of
$r_{l,\iota}(\pi)$ to a
homomorphism $$r'_{l,\iota}(\pi):G_{F^+}\to\mathcal{G}_n(\Qbar_l)$$
with $\nu\circ
r'_{l,\iota}(\pi)=\epsilon^{1-n}\delta_{F/F^+}^{\mu_\pi}$, and
$r'_{l,\iota}(\pi)(c_v)\notin\gn^{0}(\Qbar_l)$ for any infinite place
$v$ of $F^+$, where $c_{v}$ denotes complex conjugation at $v$. Here
$\delta_{F/F^+}$ is the unique nontrivial character of $\Gal(F/F^+)$,
and $\mu_\pi\in\Z/2\Z$. Accordingly, from now on we will work with
representations to $\gn$.

We now define the deformation rings we work with, following
\cite{cht}, \S 2.2. Suppose from now on that every place of $F^+$
dividing $l$ splits in $F$, and suppose that $l>n$ and $l$ is
unramified in $F^{+}$. Take $K/\Q_l$ finite with
$\#\Hom(F^+,K)=[F^+:\Q]$. Let $\bigO=\bigO_K$, let $\lambda$ be the
maximal ideal of $\bigO$, $k=\bigO/\lambda$. Let $\mathcal{C}_\bigO^f$
be the category of Artinian local $\bigO$-algebras $R$ for which the
map $\bigO\to R$ induces an isomorphism on residue fields. Let
$\mathcal{C}_\bigO$ be the category of complete local $\bigO$-algebras
with residue field $k$. We will assume from now on that $K$ is large
enough that all representations that we wish to consider which are
valued in finite extensions of $K$ are in fact valued in $K$; this may
be achieved by replacing $K$ by a finite extension.

Fix a continuous homomorphism $$\overline{r}:G_{F^+}\to\gn(k)$$ with
$\rbar^{-1}(\GL_n(k)\times\GL_1(k))=G_F$. Let $\rhobar:=\rbar|_{G_F}$. Assume that $\rbar|_{G_F}$
is absolutely irreducible and automorphic of weight $a$. To be precise, suppose that
$\rbar=\rbar'_{l,\iota}(\pi)$, where $\pi$ is a RACSDC
automorphic representation of level prime to $l$ and weight $a$. % Note
% that by Lemmas \ref{5.1} and \ref{5.2}, $r_{l,\iota}(\pi)$ extends to
% a representation $r:G_{F^+}\to\gn(K)$ lifting $\rbar$, with $\nu\circ
% r=\epsilon^{n-1}\delta_{F/F^+}^{\mu_\pi}$ for some $\mu_\pi\in\Z/2\Z$,
% where $\delta_{F/F^+}$ is the quadratic character corresponding to
% $F/F^+$. 

Suppose further that for all $\tau\in(\Z^n)^{\Hom(F,\C)}$ we have
either $$l-1-n\geq a_{\tau,1}\geq\dots\geq a_{\tau,n}\geq 0$$
or $$l-1-n\geq a_{\tau c,1}\geq\dots\geq a_{\tau c,n}\geq 0.$$

If $R$ is an object of $\mathcal{C}_\bigO$ then a \emph{lifting} of $\rbar$
to $R$ is a continuous homomorphism $$r:G_{F^+}\to\gn(R)$$with
$r\text{ mod }\mathfrak{m}_R=\rbar$ and $\nu\circ r=\epsilon^{1-n}\delta_{F/F^+}^{\mu_\pi}$. A
\emph{deformation} of $\rbar$ to $R$ is a $\ker(\gn(R)\to\gn(k))$-conjugacy
class of liftings.

We wish, as usual, to impose some local conditions on the deformations we consider. Let $S_l$ be the set of places of $F^+$ above $l$. Fix $T$ a finite set of finite places of $F^+$ containing $S_l$ and all places at which $\overline{r}$ is ramified, and assume that all places in $T$ split in $F$. We choose a set $\tilde{T}$ of places of $F$ above the places in $T$ so that we may write the places of $F$ above places in $T$ as $\tilde{T}\coprod c\tilde{T}$. If $X\subset T$, write $\tilde{X}$ for the set of places in $\tilde{T}$ above places of $X$. We will sometimes identify $F^{+}_{v}$ and $F_{{\tilde{v}}}$.

Write $G_{F^+,T}=\Gal(F(T)/F^{+})$ where $F(T)$ is the maximal
extension of $F$ unramified outside of the places lying over $T$. From now on, we consider $\rbar$, $\chi$ as $G_{F^+,T}$-representations, and all our global deformations are to $G_{F^+,T}$-representations. By Proposition 1.2.9 of \cite{cht} there is a universal deformation $r^{univ}_T$ of $\rbar$ over an object $R_T^{{global}}$ of $\mathcal{C}_\bigO$. Note that since $\rbar|_{G_{F}}$ is absolutely irreducible, $H^0(G_{F^{+},T},\ad\rbar)=0$.

Write $$T=S_l\coprod Y.$$For each place $v\in T$ there is a universal lifting of $\rbar|_{G_{F_v}}$ over an object $R_{\tv}^{loc}$ of $\mathcal{C}_\bigO$. We consider the following $\ker(\GL_n(R_{\tv}^{loc})\to\GL_n(k))$-invariant quotients of $R_{\tv}^{loc}$.
\begin{enumerate}
	\item If $v\in S_l$ then $R_{\tv}$ is the maximal quotient of $R_{\tv}^{loc}$ such that for every Artinian quotient $A$ of $R_{\tv}$ the pushforward of the universal lifting of $\rbar|_{G_{F_{\tv}}}$ to $A$ is in the essential image of the Fontaine-Laffaille functor $\mathbb{G}_{\tv}$ of \S 1.4.1 of \cite{cht}.
	\item If $v\in Y$ then choose an inertial type $\tau_\tv$ of $I_{F_\tv}$ (as in \S \ref{local}), and let $R_\tv$ be the maximal quotient of $R_\tv^{loc}$ over which the universal lift of $\rbar|_\gfv$ is of type $\tau_\tv$. 
\end{enumerate}
Say that a lifting is of type $\tau$ if for each $v\in T$ it factors
through the ring $R_\tv$ constructed above. Then there is a universal
deformation $r_\tau^{univ}$ of $\rbar$ of type $\tau$ over an object
$R_\tau^{univ}$ of $\mathcal{C}_\bigO$. In the terminology of section
2.3 of \cite{cht}, this is the universal deformation ring
$R_{\mcal{S}_\tau}$ for the deformation
problem \[\mcal{S}_\tau=(F/F^+,T,\tilde{T},\bigO,\rbar,\epsilon^{1-n}\delta_{F/F^+}^{\mu_\pi},\{R_{\tilde{v}}\}_{v\in
  T}).\]

As in \S 1.5 of \cite{cht}, we say that a subgroup $H\subset\GL_{n}(k)$ is \emph{big} if:
\begin{itemize}
	\item $H$ has no $l$-power order quotients.
	\item $H^0(H,\mathfrak{g}^{0}_n(k))=(0)$, 
	\item $H^1(H,\mathfrak{g}^{0}_n(k))=(0)$,
	\item For all irreducible $k[H]$-submodules $W$ of $\mathfrak{g}_n(k)$ we can find $h\in H$ and $\alpha\in k$ satisfying the following properties. The $\alpha$-generalised eigenspace $V_{h,\alpha}$ of $h$ on $k^n$ is one-dimensional. Let $\pi_{h,\alpha}:k^n\to V_{h,\alpha}$ be the $h$-equivariant projection of $k^n$ to $V_{h,\alpha}$, and let $i_{h,\alpha}:V_{h,\alpha}\into k^n$ be the $h$-equivariant injection of $V_{h,\alpha}$ into $k^n$. Then $\pi_{h,\alpha}\circ W\circ i_{h,\alpha}\neq (0)$.
\end{itemize}We say that a subgroup $H\subset\gn(k)$ is $\emph{big}$
if $H$ surjects onto $\gn(k)/\gn^{0}(k)$ and $H\cap\gn^{0}(k)$ is big.

We assume from now on that $\overline{F}^{\ker\ad\rbar|_{G_F}}$ does not contain $F(\zeta_l)$, and that $\rbar(G_{F^+(\zeta_l)})$ is big. 

We wish to obtain a lower bound on the dimension of $R_{\tau}^{univ}$. This is straightforward, and is in fact done in \cite{cht}. 

\begin{lemma}
	\label{5.5}$\dim R_\tau^{univ}\geq
        1-\frac{n}{2}[F^+:\Q](1+(-1)^{1+n+\mu_\pi})$. In particular, if
        $\mu_\pi\equiv n\pmod{2}$, then $\dim R_\tau^{univ}\geq 1$.
\end{lemma}
\begin{proof}% By Lemma \ref{5.3}, \begin{align*}\dim R_{\tau}^{\square}&\geq\dim R_{T}^{\tau}+\dim R_{T}^{\square}-\dim R_{T}^{loc}\\
By Corollary 2.3.5 of \cite{cht}, $$\dim R_{\tau}^{univ}\geq
1+\sum_{v\in T}(\dim
R_{\tv}-n^{2}-1)-\dim_{k}H^{0}(G_{F^{+},T},\ad\rbar(1))-n\sum_{v|\infty}(n+\chi(c_{v}))/2,$$where $\chi=\epsilon^{1-n}\delta_{F/F^+}^{\mu_\pi}$.

Now, by Theorem \ref{2.6} of this paper and Corollary 2.4.3 of \cite{cht}, we see that if $v\nmid l$ then $\dim R_{\tv}=1+n^{2}$, and if $v|l$ then $\dim R_{\tv}=1+n^{2}+\frac{n(n-1)}{2}[F_{\tv}:\Q_{l}]$. Thus \begin{align*}\sum_{v\in T}(\dim R_{\tv}-n^{2}-1)&=\frac{n(n-1)}{2}[F:\Q]\\&=n\sum_{v|\infty}(n-1)/2.\end{align*}The assumption that $\rbar(G_{F^+(\zeta_l)})$ is big implies that $H^{0}(G_{F^{+},T},\ad\rbar(1))=0$, and the result follows.\end{proof}

In fact,  $\mu_\pi\equiv n\pmod{2}$; this is obtained in \cite{guerberoff2009modularity} as a consequence of the Taylor-Wiles method.

We wish to use the $R=T$ theorems proved in
\cite{guerberoff2009modularity} to show that $(R_\tau^{univ})^{red}$
is finite over $\bigO$.  We now make a base change, as in the proof of
Theorem 5.2 of \cite{tay06}. Enlarge $Y$ if necessary, so that
$\tilde{T}\cup c\tilde{T}$ contains all the places at which $\pi$ is
ramified. Because $\overline{F}^{\ker\ad\rbar|_{G_F}}$ does not
contain $F(\zeta_l)$, we can choose a finite place $v_1$ of $F^{+}$
such that:
\begin{itemize}
	\item $v_1\notin T$, and $v_{1}$ splits in $F$. 
	\item $v_1$ is unramified over a rational prime $p$ for which $[F(\zeta_p):F]>n$.
	\item $v_1$ does not split completely in $F(\zeta_l)$.
	\item $\ad\rbar(\Frob_{v_1})=1$.
\end{itemize}
Now choose a totally real field $L^+/F^+$ satisfying the following conditions:
\begin{itemize}
	\item $4|[L^+:\Q]$. 
	\item $L^+/F^+$ is Galois and solvable. 
	\item $L=L^+ F$ is linearly disjoint from $\overline{F}^{\ker\rbar|_{G_F}}(\zeta_l)$ over $F$.
	\item $L/L^+$ is everywhere unramified.
	\item $l$ is unramified in $L^+$.
	\item $v_1$ splits completely in $L/F^{+}$. 
	\item Let $\pi_L$ denote the base change of $\pi$ to $L$. If
          $w$ is a place of $L$ lying over $v\in Y$ then
          $(\pi_{L,w})^{Iw(w)}\neq(0)$ and ${\tau_v}|_{I_{L_{w}}}$ is
          trivial (here $Iw(w)$ is the Iwahori subgroup of
          $\GL_n(\bigO_{L,w})$). In addition, $\operatorname{N}w\equiv 1$ mod $l$ and $\rbar|_{G_{L_w}}$ is trivial.
\end{itemize}

Let $S_1(L^+)$ be the places of $L^+$ above
$v_1$. Let $S_l(L^+)$ denote the places of $L^+$
above $l$. Let $Y(L^+)$ be the places of $L^+$ above elements of
$Y$. Then let $T(L^+)=S_1(L^+)\cup S_l(L^+)\cup Y(L^+)$,
and choose a set $\tilde{T}(L)$ consisting of a choice of a place of
$L$ above every place of $T(L^+)$, in such a way that $\tilde{T}(L)$
contains all places of $L$ lying over places in $\tilde{T}$.

Let $a_L\in(\Z^n)^{\Hom(L,\Qbar_l)}$ be defined by
$a_{L,\tau}=a_{\tau|_F}$, so that $\rbar|_{G_L}$ is automorphic of
weight $a_L$.

We now consider certain deformations of $\rbar|_{G_{L^+}}$. Note that $(\nu\circ\rbar)|_{G_{L^+}}=\epsilon^{1-n}\delta_{L/L^+}^{\mu_\pi}$, with $\mu_\pi$ as above, and $\delta_{L/L^+}$ the unique non-trivial character of $\Gal(L/L^+)$. A lifting of $\rbar|_{G_{L^+}}$ to an object $R$ of $\mathcal{C}_\bigO$ will be a continuous homomorphism $r:G_{L^+}\to\gn(R)$ with $r\text{ mod }\mathfrak{m}_R=\rbar$ and $\nu\circ r=\epsilon^{n-1}\delta_{L/L^+}^{\mu_\pi}$. For each $v\in T(L^+)$ there is a universal lifting of $\rbar|_{G_{L_\tv}}$ over an object $R_\tv^{loc}$. We consider the following $\ker(\GL_n(R_\tv^{loc})\to\GL_n(k))$-invariant quotients $R_\tv$ of $R_\tv^{loc}$.
\begin{enumerate}
	\item If $v\in S_l(L^+)$ then $R_{\tv}$ is the maximal quotient of $R_{\tv}^{loc}$ such that for every Artinian quotient of $R_{\tv}$ the pushforward of the universal lifting of $\rbar|_{G_{F_{\tv}}}$ to $A$ is in the essential image of the Fontaine-Laffaille functor $\mathbb{G}_{\tv}$ of \S 1.4.1 of \cite{cht}. 
	\item If $v\in Y(L^+)$ then let $R_\tv$ be the maximal quotient of $R_\tv^{loc}$ over which for all $\sigma\in I_{L_\tv}$ the universal lift of $\rbar|_{G_{L_\tv}}$ evaluated at $\sigma$ has characteristic polynomial $(X-1)^n$. 
	\item If $v\in S_1(L^+)$ then $R_\tv=R_\tv^{loc}$. 
\end{enumerate} Say that a lifting of $\rbar|_{G_{L^{+}}}$ is of type $\mathcal{S}$ if
it is unramified ouside $T(L^+)$ and if at each $v\in T(L^+)$ it
factors through the ring $R_\tv$ defined above. There is a universal
deformation $r_{\mathcal{S}}^{univ}$ of type $\mathcal{S}$ over an
object $R_\mathcal{S}^{univ}$ of $\mathcal{C}_\bigO$. As above, in the
terminology of \cite{cht} this is the universal deformation ring for
the deformation problem  \[\mcal{S}=(L/L^+,T(L^+),\tilde{T}(L),\bigO,\rbar|_{G_{L^+}},\epsilon^{1-n}\delta_{L/L^+}^{\mu_\pi},\{R_{\tilde{v}}\}_{v\in
  T(L^+)}).\]Now, if we choose a representative for the universal
deformation $r^{univ}_{\mcal{S_\tau}}:G_{F^+}\to\gn(R_\tau^{univ})$,
and restrict it to $G_{L^+}$, the universal property gives
$R_\tau^{univ}$ the structure of an $R_{\mcal{S}}^{univ}$-module.
\begin{thm}
	\label{5.7}With the above assumptions, $(R_\tau^{univ})^{red}$ is a finite $\bigO$-module of rank at least $1$.
\end{thm}
\begin{proof}
  By Theorem 3.4 of \cite{guerberoff2009modularity}, $\mu_\pi\equiv
  n\text{ mod }2$, so that Corollary \ref{5.5} shows that $\dim
  (R_\tau^{univ})^{red}=\dim R_\tau^{univ}\geq 1$. As in the proof of
  Proposition \ref{1.5}, it suffices to show that
  $(R_\tau^{univ})^{red}$ is a finite $\bigO$-algebra. We claim that
  $(R_\tau^{univ})^{red}$ is a finite
  $(R_\mathcal{S}^{univ})^{red}$-module; this follows just as in
  Theorem 4.2.8 of \cite{kiscdm} and Lemma 3.6 of \cite{kw}, using
  Lemma 2.1.12 of \cite{cht}. 

More precisely, write $\mf{m}_{\cS}$ for the maximal ideal of
$R_\mathcal{S}^{univ}$, and let
$\bar{R}:=R_\tau^{univ}/\mf{m}_\cS$. Let $\rbar_{\tau,\cS}$ denote the
$\bar{R}$-representation obtained from the universal representation
over $R_\tau^{univ}$. Since  $\rbar_{\tau,\cS}|_{G_{L^+,T(L^+)}}$ is
  equivalent to $\rbar|_{G_{L^+,T(L^+)}}$, we see that
    $\rbar_{\tau,\cS}$ has finite image. By Lemma 2.1.12 of
    \cite{cht}, $\bar{R}$ is generated by the traces of the
    $\rbar_{\tau,\cS}(g)$, $g\in G_{F,T}$. Let $\mf{p}$ be a
    prime ideal of $\bar{R}$. Since the image of $\rbar_{\tau,\cS}$ is
    finite, the images of the traces  of the
    $\rbar_{G_{L^+,T(L^+)}}(g)$ in $\bar{R}/\mf{p}$ are sums of roots
    of unity of bounded degree, so $\bar{R}/\mf{p}$ is finite, and
    thus a field. So $\bar{R}$ is a 0-dimensional Noetherian ring, so
    it is Artinian, and thus a finite product of Artin local rings
    with finite residue fields, so it is finite. Since
    $R_\tau^{univ}/\mf{m}_\cS$ is finite, $R_\tau^{univ}$ is finite
    over $R_\mathcal{S}^{univ}$ by the topological version of
    Nakayama's lemma, so  $(R_\tau^{univ})^{red}$ is a finite
  $(R_\mathcal{S}^{univ})^{red}$-module, as claimed.

 Thus we need only check that
  $(R_\mathcal{S}^{univ})^{red}$ is finite over $\bigO$. By Theorem
  3.4 of \cite{guerberoff2009modularity} we have $(R_\mathcal{S}^{univ})^{red}\isoto
  \mathbb{T}$, where $\mathbb{T}$ is a certain Hecke algebra, which is
  finite over $\bigO$ by construction. The result follows.
\end{proof}
From this, together with Theorem 3.4 of \cite{guerberoff2009modularity}, one immediately obtains the following result, where for the convenience of the reader we have incorporated all the assumptions made into the statement of the theorem.
\begin{thm}
  \label{5.8}Let $F=F^+E$ be a CM field, where $F^+$ is totally real
  and $E$ is imaginary quadratic. Let $n\geq 1$ and let $l>n$ be a
  prime which is unramified in $F^+$ and split in $E$. Suppose
  that $$\rhobar:G_F\to\GL_n(\Fbar_l)$$is an irreducible
  representation which is unramified at all places of $F$ lying above
  primes which do not split in $E$, and which satisfies the following
  properties.
	\begin{enumerate}
        \item  $\rhobar$
          is automorphic of weight $a$, where we
          assume that for all $\tau\in(\Z^n)^{\Hom(F,\C)}$ we have
          either $$l-1-n\geq a_{\tau,1}\geq\dots\geq a_{\tau,n}\geq
          0$$ or $$l-1-n\geq a_{c\tau,1}\geq\dots\geq a_{c\tau,n}\geq
          0.$$ Note than in particular these conditions imply that
          $\rhobar^{c}\cong\rhobar^{\vee}\epsilon^{1-n}$.
       
		\item $\overline{F}^{\ker\ad\rhobar}$ does not contain $F(\zeta_l)$. 
		\item $\rhobar(G_{F^+(\zeta_l)})$ is big.
	\end{enumerate}
	Let $Y$ be a finite set of finite places of $F^+$ which split
        in $F$, which does not contain any places dividing
        $l$. For each $v\in Y$, choose an inertial type $\tau_v$ and a
        place $\tv$ of $F$ above $v$. Assume that $\rhobar|_\gfv$ has
        a lift to characteristic zero of type $\tau_v$.
	
	Then there is an automorphic representation $\pi$ on $\GL_n(\A_F)$ of weight $a$ and level prime to $l$ such that 
	\begin{enumerate}
		\item $\rbar_{l,\iota}(\pi)\cong\rhobar$. 
		\item $r_{l,\iota}(\pi)|_\gfv$ has type $\tau_v$ for all $v\in Y$. 
		\item $\pi$ is unramified at all places of $F$ at which $\rhobar$ is unramified, except possibly for the places lying over elements of $Y$.
	\end{enumerate}
\end{thm}

\bibliographystyle{amsalpha} %Bibliography style file X.bst
\bibliography{tobybib08} % Bibliography database file Y.bib

\end{document}